\documentclass[reqno,11pt]{article}
\usepackage{a4wide,color,eucal,enumerate,mathrsfs}
\usepackage[normalem]{ulem}
\usepackage{pdfsync}
\usepackage[color]{}
\usepackage{amsmath,amssymb,epsfig,bbm}
\numberwithin{equation}{section}

\usepackage[latin1]{inputenc}

\newcommand{\C}{\mathbb{C}}

\newcommand{\N}{\mathbb{N}}

\newcommand{\R}{\mathbb{R}}

\newcommand{\mm}{{\mbox{\boldmath$m$}}}

\newcommand{\smm}{{\mbox{\scriptsize\boldmath$m$}}}

\newcommand{\ggamma}{{\mbox{\boldmath$\gamma$}}}

\newcommand{\ppi}{{\mbox{\boldmath$\pi$}}}

\newcommand{\ssigma}{{\mbox{\boldmath$\sigma$}}}

\newcommand{\sggamma}{{\mbox{\scriptsize\boldmath$\gamma$}}}

\newcommand{\sfd}{{\sf d}}

\newcommand{\rme}{{\mathrm e}}

\newcommand{\Kliminf}{K\kern-3pt-\kern-2pt\mathop{\rm lim\,inf}\limits}  
\newcommand{\supp}{\mathop{\rm supp}\nolimits}   
\newcommand{\Lip}{\mathop{\rm Lip}\nolimits}          
\renewcommand{\d}{{\mathrm d}}
\newcommand{\dt}{{\d t}}

\newcommand{\restr}[1]{\lower3pt\hbox{$|_{#1}$}}

\newcommand{\Leb}[1]{{\mathscr L}^{#1}}      
\newcommand{\la}{{\langle}}                  
\newcommand{\ra}{{\rangle}}

\newcommand{\eps}{\varepsilon}  
\newcommand{\nchi}{{\raise.3ex\hbox{$\chi$}}}

\newcommand{\forevery}{\text{for every }}

\newcommand{\Pc}[2]{\overline{#1}\kern-2pt^{\vphantom 0}_{#2}}

\newcommand{\Probabilities}[1]{\mathscr P(#1)}          
\newcommand{\ProbabilitiesTwo}[1]{\mathscr P_2(#1)}     

\newenvironment{proof}{\removelastskip\par\medskip   
\noindent{\em Proof}
\rm}{\penalty-20\null\hfill$\square$\par\medbreak}

\newtheorem{theorem}{Theorem}[section]

\newtheorem{lemma}[theorem]{Lemma}
\newtheorem{proposition}[theorem]{Proposition}

\newtheorem{definition}[theorem]{Definition}

\newtheorem{example}[theorem]{Example}
\newtheorem{remark}[theorem]{Remark}

\newcommand{\ent}[1]{{\rm Ent}_{\smm}(#1)}
\newcommand{\entv}{{\rm Ent}_{\smm}}
\newcommand{\entr}[2]{{\rm Ent}_{#2}(#1)}
\newcommand{\prob}{\Probabilities}
\newcommand{\probt}{\ProbabilitiesTwo}

\newcommand{\adm}{\textrm{\sc{Adm}}}
\newcommand{\opt}{\textrm{\sc{Opt}}}

\newcommand{\Id}{{\rm Id}}

\newcommand{\lims}{\varlimsup}
\newcommand{\limi}{\varliminf}
\newcommand{\geo}{\rm Geo}
\newcommand{\gopt}{{\rm{GeoOpt}}}                   
\newcommand{\e}{{\rm{e}}}                           
\newcommand{\fr}{\hfill$\blacksquare$}                      
\newcommand{\sppi}{{\mbox{\scriptsize\boldmath$\pi$}}}      

\newcommand{\relgrad}[1]{|D #1|_*} 
\newcommand{\weakgrad}[1]{|D #1|_w} 
\newcommand{\cgrad}[1]{|D #1|_C} 

\renewcommand{\mm}{\mathfrak m}
\renewcommand{\smm}{{\mbox{\scriptsize$\mm$}}}

\newcommand{\compression}{deformation}

\newcommand{\X}{{Y}}
\newcommand{\sfdX}{\sfd_\X}
\newcommand{\x}{y}
\newcommand{\frho}{f}
\newcommand{\hsigma}{h}
\newcommand{\intX}{\int_X}

\newcommand{\GGG}{\color{blue}}

\renewcommand{\C}{\mathsf{Ch}}

\title{Heat flow and calculus on metric measure spaces with \\ Ricci curvature bounded below - the compact case
\footnote{To the memory of Enrico Magenes, whose exemplar life, research
and teaching shaped generations of mathematicians}}

\begin{document}
\author{Luigi Ambrosio\
   \thanks{Scuola Normale Superiore, Pisa. email: \textsf{l.ambrosio@sns.it}}
   \and
   Nicola Gigli\
   \thanks{Nice University. email: \textsf{nicola.gigli@unice.fr}}
 \and
   Giuseppe Savar\'e\
   \thanks{Universit\`a di Pavia. email: \textsf{giuseppe.savare@unipv.it}
   }}

\maketitle

\begin{abstract}
 We provide a quick overview of various calculus tools and of the 
 main results concerning the heat
 flow  on compact
 metric measure spaces,
 with applications to spaces with lower Ricci curvature bounds.

 Topics include the  Hopf-Lax semigroup and the Hamilton-Jacobi equation
 in metric spaces, a new approach to differentiation and to the theory of Sobolev spaces over
 metric measure spaces, 
 the equivalence of the $L^2$-gradient flow of a suitably defined 
 ``Dirichlet energy'' and the Wasserstein gradient flow of the
 relative entropy functional, a metric version of Brenier's Theorem, 
and a new (stronger) definition of Ricci curvature bound from below
for metric measure spaces. This new notion is stable w.r.t. measured
Gromov-Hausdorff convergence and 
it is strictly connected with the linearity of the heat flow.
\end{abstract}
\tableofcontents

\section{Introduction}

Aim of these notes is to provide a quick overview of the main
results contained in \cite{Ambrosio-Gigli-Savare11} and
\cite{Ambrosio-Gigli-Savare11bis} in the simplified case of compact
metric spaces $(X,\sfd)$ endowed with a reference probability
measure $\mm$. The idea is to give the interested reader the
possibility to get as quickly as possible the key ideas behind the
proofs of our recent results, neglecting all the problems that
appear in a more general framework (as a matter of fact, no
compactness assumption is made in
\cite{Ambrosio-Gigli-Savare11,Ambrosio-Gigli-Savare11bis} and
finiteness of $\mm$ is assumed only in
\cite{Ambrosio-Gigli-Savare11bis}). Passing from compact spaces to complete and separable ones (and even
to a more general framework which includes the so-called Wiener
space) is not just a technical problem, meaning that several
concepts need to be properly adapted in order to achieve such
generality.  Hence, in
particular, the discussion here is by no means exhaustive, as both
the key statements and the auxiliary lemmas are stated in the
simplified case of a probability measure in a compact space.

Apart some very basic concept about optimal transport, Wasserstein
distance and gradient flows, this paper pretends to be
self-contained. All the concepts that we need are recalled in the
preliminary section, whose proofs can be found, for instance, in the
first three chapters of \cite{Ambrosio-Gigli11} (for an overview on
the theory of gradient flows, see also
\cite{Ambrosio-Gigli-Savare08}, and for a much broader discussion on
optimal transport, see the monograph by Villani \cite{Villani09}).
For completeness reasons, we included in our discussion some results
coming from previous contributions which are potentially less known,
in particular: the (sketch of the) proof by Lisini \cite{Lisini07}
of the characterization of absolutely continuous curves w.r.t. the
Wasserstein distance (Proposition~\ref{prop:lisini}), and the proof
of uniqueness of the gradient flow of the relative entropy w.r.t.
the Wasserstein distance on spaces with Ricci curvature bounded
below in the sense of Lott-Sturm-Villani ($CD(K,\infty)$ spaces in
short) given by the second author in \cite{Gigli10}
(Theorem~\ref{thm:gfent}).

In summary, the main arguments and results that we present here are the following.
\begin{itemize}
\item[(1)] The Hopf-Lax formula produces subsolutions of the Hamilton-Jacobi
equation, and solutions on geodesic spaces (Theorem~\ref{thm:subsol}
and Theorem~\ref{thm:supersol}).
\item[(2)] A new approach to the theory of Sobolev spaces over metric measure spaces, which leads in particular
to the proof that \emph{Lipschitz functions are always dense in
energy in $W^{1,2}(X,\sfd,\mm)$} (Theorem~\ref{thm:lipdense}).
\item[(3)] The uniqueness of the
  gradient flow w.r.t.\ the Wasserstein distance $W_2$ of the relative entropy
in $CD(K,\infty)$ spaces (Theorem~\ref{thm:gfent}).
\item[(4)] The identification of the $L^2$-gradient flow of the natural ``Dirichlet energy'' and the $W_2$-gradient flow
of the relative entropy in $CD(K,\infty)$ spaces (see also
\cite{GigliKuwadaOhta10} for the Alexandrov case, a paper to which
our paper \cite{Ambrosio-Gigli-Savare11} owes a lot).
\item[(5)] A metric version of Brenier's theorem valid in spaces having Ricci curvature bounded from below
in a sense slightly stronger than the one proposed by
Lott-Sturm-Villani. 
 If this curvature
assumption holds (Definition~\ref{def:strongcd}) and $\mu,\,\nu$ are
absolutely continuous w.r.t. $\mm$, then ``the distance traveled is
uniquely determined by the starting point'', i.e. there exists a map
$D:X\to\R$ such that for any optimal plan $\ggamma$ it holds
$\sfd(x,y)=D(x)$ for $\ggamma$-a.e.~$(x,y)$. Moreover, the map $D$ is
nothing but the weak gradient (according to the theory illustrated
in Section~\ref{se:Sobolev}) of any Kantorovich potential. See
Theorem~\ref{prop:potweak}.
\item[(6)] A key lemma (Lemma \ref{le:horver}) concerning ``horizontal'' and ``vertical'' differentiation:
it allows to compare the derivative of the squared Wasserstein
distance along the heat flow with the derivative of the
relative entropy along a geodesic.
\item[(7)] A new (stronger) definition of Ricci curvature bound from below for
  metric measure spaces which is
stable w.r.t. measured Gromov-Hausdorff
convergence and rules out Finsler geometries
(Theorem~\ref{thm:riemannian} and the discussion thereafter).
\end{itemize}

\smallskip
\noindent {\bf Acknowledgement.} The authors acknowledge the support
of the ERC ADG GeMeThNES and
the PRIN08-grant from MIUR for the project \emph{Optimal transport
  theory, geometric and functional inequalities, and
  applications}.
  
The authors also thank A.Mondino for his
careful reading of a preliminary version of this manuscript.

\section{Preliminary notions}\label{sec:preliminary}

As a general convention, we will always denote 
by $(X,\sfd)$ 
a compact metric space and by $\mm$ a Borel probability measure on $X$;
we will always refer to the structure $(X,\sfd,\mm)$ as a compact and normalized
metric measure space. We will use the symbol $(Y,\sfd_Y)$ for metric spaces
when the compactness is not implicitly assumed.

\subsection{Absolutely continuous curves and slopes}

Let $(\X,\sfdX)$ be a complete and separable metric space,
 $J\subset\R$ an interval with
nonempty interior and $J\ni t\mapsto \gamma_t\in \X$. We say that $\gamma_t$ is
\emph{absolutely continuous} if
$$
\sfdX(\gamma_s,\gamma_t)\leq\int_t^sg(r)\,\d r,\qquad\forall s,\,t\in J,\,\,t<s
$$
for some $g\in L^1(J)$. It turns out that, if $\gamma_t$ is absolutely
continuous, there is a minimal function $g$ with this property,
called \emph{metric speed} and given for a.e.~$t\in J$ by
$$
|\dot{\gamma_t}|=\lim_{s\to t}\frac{\sfdX(\gamma_s,\gamma_t)}{|s-t|}.
$$
See \cite[Theorem~1.1.2]{Ambrosio-Gigli-Savare08} for the simple
proof. Notice that the absolute continuity property of the
integral ensures that absolutely continuous functions can be extended
by continuity to the closure of their domain.

We will denote by $C([0,1],\X)$ the space of continuous curves on
$[0,1]$ with values in $\X$ endowed with the $\sup$ norm. The set
$AC^2([0,1],\X)\subset C([0,1],\X)$ consists of all absolutely
continuous curves $\gamma$ such that $\int_0^1|\dot\gamma_t|^2\,\d
t<\infty$: it is easily seen to be equal to the countable union of
the closed sets $\{\gamma:\int_0^1|\dot\gamma_t|^2\,\d t\leq n\}$,
and thus it is a Borel subset of $C([0,1],\X)$. The \emph{evaluation
maps} $\e_t:C([0,1],\X)\to \X$ are defined by
\[
\e_t(\gamma):=\gamma_t,
\]
and are clearly $1$-Lipschitz.

We say that a subset $D$ of $\X$  is geodesic 
if for any $x,\,y\in D$
there exists a curve $(\gamma_t) \subset D$ on $[0,1]$ such that $\gamma_0=x$,
$\gamma_1=y$ and $\sfdX(\gamma_t,\gamma_s)=|t-s|\sfdX(x,y)$ for all
$s,\,t\in [0,1]$. Such a curve is called constant speed geodesic, or
simply geodesic. The space of all geodesics in $\X$ endowed with the sup
distance will be denoted by $\geo(\X)$.

Given $f:\X\to\R\cup\{\pm\infty\}$ we define the \emph{slope} (also
called local Lipschitz constant) at points $x$ where $f(x)\in\R$ by
$$
|D f|(x):=\lims_{y\to x}\frac{|f(y)-f(x)|}{\sfdX(y,x)}.
$$

We shall also need the one-sided counterparts of the slope  called respectively \emph{descending slope} and
\emph{ascending slope}:
\begin{equation}\label{eq:slopes}
|D^- f|(x):=\lims_{y\to x}\frac{[f(y)-f(x)]^-}{\sfdX(y,x)},\qquad
|D^+ f|(x):=\lims_{y\to x}\frac{[f(y)-f(x)]^+}{\sfdX(y,x)},
\end{equation}
where $[\cdot]^+$ and $[\cdot]^-$ denote respectively the positive and negative part. Notice the change of notation w.r.t. previous works of the authors: the slopes and its one-sided counterparts were denoted by $|\nabla f|$, $|\nabla^\pm f|$. Yet, as remarked in \cite{Gigli12bis}, these notions, being defined in duality with the distance, are naturally cotangent notions, rather than tangent ones, whence the notation proposed here.

It is not difficult to see that for $f$ Lipschitz the  slopes and
the local Lipschitz constant are upper gradients according to
\cite{Heinonen-Koskela98}, namely
$$
\left| \int_{\partial\gamma} f\right|
\leq\int_\gamma|D^\pm f|
$$
for any absolutely continuous curve $\gamma:[0,1]\to \X$;
here and in the following we write 
$\int_{\partial\gamma}f$ for $f(\gamma_1)-f(\gamma_0)$ and
$\int_\gamma g$ for $\int_0^1
g(\gamma_s)|\dot \gamma_s|\,\d s.$

Also, for $f,\,g:\X\to\R$ Lipschitz it clearly holds
\begin{subequations}
\begin{align}
\label{eq:subadd}
|D(\alpha f+\beta g)|&\leq|\alpha||D f|+|\beta||D g|,\qquad\forall \alpha,\beta\in\R;\\
\label{eq:leibn} |D (fg)|&\leq |f||D g|+|g||D f|.
\end{align}
\end{subequations}
\subsection{The space $(\prob X,W_2)$}

Let $(X,\sfd)$ be a compact metric space. The set $\prob X$ consists
of all Borel probability measures on $X$. 
As usual, if $\mu\in \prob X$ and $T : X \to Y$ is a $\mu$-measurable map
with values in the topological space $Y$, 
the push-forward measure $T_\sharp \mu\in \prob Y$ is defined by
$T_\sharp \mu(B) := \mu(T^{-1}(B))$ for every set Borel set $B\subset
Y$.

Given $\mu,\,\nu\in\prob
X$, we define the Wasserstein distance $W_2(\mu,\nu)$ between them
as
\begin{equation}
W_2^2(\mu,\nu):=\min\int \sfd^2(x,y)\,\d\ggamma(x,y),\label{eq:1}
\end{equation}
where the minimum is taken among all Borel probability measures
$\ggamma$ on $X^2$ such that
\[
\pi^1_\sharp \ggamma=\mu,\qquad
\pi^2_\sharp \ggamma=\nu;\qquad
\text{here }\pi^i:X^2\to X,\quad \pi^i(x_1,x_2):=x_i.
\]
Such measures are called admissible plans or couplings for the couple
$(\mu,\nu)$; 
a plan $\ggamma$ which realizes the minimum in \eqref{eq:1} is
called {optimal}, and we write $\ggamma\in\opt(\mu,\nu)$. From the
linearity of the admissibility condition we get that the squared
Wasserstein distance is convex, i.e.:
\begin{equation}
\label{eq:w2conve} W_2^2\big((1-\lambda)\mu_1+\lambda
\nu_1,(1-\lambda)\mu_2+\lambda\nu_2\big)\leq(1-\lambda)
W_2^2(\mu_1,\nu_1)+\lambda W_2^2(\mu_2,\nu_2).
\end{equation}
It is also well known (see e.g. Theorem 2.7 in
\cite{Ambrosio-Gigli11}) that the Wasserstein distance metrizes the
weak convergence of measures in $\prob X$, i.e.\ the weak convergence with respect
to the duality with $C(X)$; in particular $(\prob
X,W_2)$ is a compact metric space.

An equivalent definition of $W_2$ comes from the dual formulation of the transport problem:
\begin{equation}
\label{eq:dualitabase}
\frac12W_2^2(\mu,\nu)=\sup_{\psi}\intX \psi\, \d\mu+\intX \psi^c\,\d\nu,
\end{equation}
the supremum being taken among all Lipschitz functions $\psi$, where
the $c$-transform in this formula is defined by
\[
\psi^c(y):=\inf_{x\in X}\frac{\sfd^2(x,y)}2-\psi(x).
\]
A function $\psi:X\to\R$ is said to be $c$-concave if $\psi=\phi^c$
for some $\phi:X\to\R$. 
It is possibile to prove that 
the supremum in \eqref{eq:dualitabase} is
always achieved by a $c$-concave function, and we will call any such
function $\psi$ a Kantorovich potential. We shall also use the fact
that $c$-concave functions satisfy
\begin{equation}\label{eq:involution}
\psi^{cc}=\psi.
\end{equation}

The (graph of the) $c$-superdifferential $\partial^c\psi$ of a
$c$-concave function $\psi$ is the subset of $X^2$ defined by
\[
\partial^c\psi:=\Big\{(x,y)\ :\ \psi(x)+\psi^c(y)=\frac{\sfd^2(x,y)}2\Big\},
\]
and the $c$-superdifferential $\partial^c\psi(x)$ at $x$ is the set
of $y$'s such that $(x,y)\in\partial^c\psi$. A consequence of the
compactness of $X$ is that any $c$-concave function $\psi$ is
Lipschitz and that the set $\partial^c\psi(x)$ is non empty for any
$x\in X$.

It is not difficult to see that if $\psi$ is a Kantorovich potential for
$\mu,\nu\in \prob X$ and $\ggamma$ is a coupling for $(\mu,\nu)$ then
$\ggamma$ is optimal if and only if
$\supp(\gamma)\subset \partial^c\psi$.

If $(X,\sfd)$ is geodesic,
then so is $(\prob X,W_2)$, and in this case a curve $(\mu_t)$ is a
constant speed geodesic from $\mu_0$ to $\mu_1$ if and only if there
exists a measure $\ppi\in\prob{C([0,1],X)}$ concentrated on
$\geo(X)$ such that $(\e_t)_\sharp\ppi=\mu_t$ for all $t\in [0,1]$
and $(\e_0,\e_1)_\sharp\in\opt(\mu_0,\mu_1)$. We will denote the set
of such measures, called optimal geodesic plans, by
$\gopt(\mu_0,\mu_1)$.

\subsection{Geodesically convex functionals and their gradient flows}\label{se:prelGF}

Given a geodesic space $(\X,\sfd_\X)$ (in the following this will
always be the Wasserstein space built over a geodesic space
$(X,\sfd)$), a functional $E:\X\to\R\cup\{+\infty\}$ is said
$K-$geodesically convex (or simply $K$-convex) if for any
$y_0,\,y_1\in \X$ there exists a constant speed geodesic
$\gamma:[0,1]\to \X$ such that $\gamma_0=y_0$, $\gamma_1=y_1$ and
\[
E(\gamma_t)\leq (1-t)E(y_0)+tE(y_1)-\frac
K2t(1-t)\sfd_\X^2(y_0,y_1),\qquad\forall t\in[0,1].
\]
We will denote by $D(E)$ the domain of $E$ i.e.
$D(E):=\{y:E(y)<\infty\}$: if $E$ is $K-$geodesically convex,
then
$D(E)$ is geodesic.

An easy consequence of the $K$-convexity is the fact that the
descending slope defined in \eqref{eq:slopes} can de computed as a
sup, rather than as a limsup:
\begin{equation}
\label{eq:slopegeodetiche} |D^-E|(y)=\sup_{z\neq
y}\left(\frac{E(y)-E(z)}{\sfd_\X(y,z)}+\frac{K}{2}\sfd_\X(y,z)\right)^+.
\end{equation}

What we want to discuss here is the definition of gradient flow of a
$K$-convex functional. There are essentially two different ways of
giving such a notion in a metric setting. The first one, which we
call Energy Dissipation Equality (EDE), ensures existence for any
$K$-convex and lower semicontinuous functional (under suitable
compactness assumptions), the second one, which we call Evolution
Variation Inequality (EVI), ensures uniqueness and $K$-contractivity
of the flow. However, the price we pay for these stronger properties
is that existence results for EVI solutions hold under much more
restrictive assumptions.

It is important to distinguish the two notions. The EDE one is the
``correct one'' to be used in a general metric context, because it
ensures existence for any initial datum in the domain of the
functional. However, typically gradient flows in the EDE sense are
not unique: this is the reason of the analysis made in
Section~\ref{se:gfent}, which ensures that for the special case of
the entropy functional uniqueness is indeed true.

EVI gradient flows are in particular gradient flows in the EDE sense
(see Proposition \ref{prop:giuseppe}), ensure  uniqueness,
$K$-contractivity and provide strong a priori regularizing effects.
Heuristically speaking, existence of gradient flows in the EVI sense
depends also on properties of the distance, rather than on
properties of the functional only. A more or less correct way of
thinking at this is: gradient flows in the EVI sense exist if and
only if the distance is Hilbertian on small scales.  For instance,
if the underlying metric space is an Hilbert space, then the two
notions coincide.

Now recall that one of our goals here is to study the gradient flow
of the relative entropy in spaces with Ricci curvature bounded
below (Definition~\ref{def:cdkinfty}), and recall that Finsler
geometries are included in this setting (see page 926 of
\cite{Villani09}). Thus, in general we must deal with the EDE notion
of gradient flow. The EVI one will come into play in
Section~\ref{se:riemricc}, where we use it to identify those spaces
with Ricci curvature bounded below which are more `Riemannian like'.

\noindent {\bf Note}: later on we will refer to gradient flows in
the EDE sense simply as ``gradient flows'', keeping the
distinguished notation ${\rm EVI}$-gradient flows for those in the
EVI sense.

\subsubsection{Energy Dissipation Equality}

An important property of $K$-geodesically convex and lower
semicontinuous functionals  (see Corollary 2.4.10 of
\cite{Ambrosio-Gigli-Savare08} or Proposition 3.19 of \cite{Ambrosio-Gigli11}) is that the descending slope is an
upper gradient, that is: for any absolutely continuous curve
$y_t:J\subset\R\to D(E)$ it holds
\begin{equation}
\label{eq:slopestrongupper} |E(y_t)-E(y_s)|\leq \int_t^s |\dot
y_r||D^- E|(y_r)\,\d r,\qquad\forall t\leq s.
\end{equation}
An application of Young inequality gives that
\begin{equation}
\label{eq:boundtuttecurve}
E(y_t)\leq E(y_s)+\frac12\int_t^s|\dot y_r|^2\,\d r
+\frac12\int_t^s|D^- E|^2(y_r)\,\d r,\qquad\forall t\leq s.
\end{equation}
This inequality motivates the following definition:
\begin{definition}[Energy Dissipation Equality definition of gradient flow]\label{def:ede}
Let $E$ be a $K$-convex and lower semicontinuous functional and let
$y_0\in D(E)$.  We say that a continuous curve $[0,\infty)\ni
t\mapsto y_t$ is a gradient flow for the $E$ in the EDE sense (or
simply a gradient flow) if it is locally absolutely continuous in
$(0,\infty)$, it takes values in the domain of $E$ and it holds
\begin{equation}
\label{eq:defede}
E(y_t)= E(y_s)+\frac12\int_t^s|\dot y_r|^2\,\d r+\frac12\int_t^s|D^- E|^2(y_r)\,\d r,\qquad\forall t\leq s.
\end{equation}
\end{definition}
Notice that due to \eqref{eq:boundtuttecurve} the equality \eqref{eq:defede} is equivalent to
\begin{equation}
\label{eq:defedegeq}
E(y_0)\geq E(y_s)+\frac12\int_0^s|\dot y_r|^2\,\d r+\frac12\int_0^s|D^- E|^2(y_r)\,\d r,\qquad\forall s>0.
\end{equation}
Indeed, if \eqref{eq:defedegeq} holds, then \eqref{eq:defede} holds
with $t=0$, and then by linearity \eqref{eq:defede} holds in
general.

It is not hard to check that if $E:\R^d\to\R$ is a $C^1$ 
function, 
then a curve $y_t:J\to\R^d$ is a gradient flow according
to the previous definition if and only if it satisfies
\[
y_t'=-\nabla E(y_t),\qquad \forall t\in J,
\]
so that the metric definition reduces to the classical one when specialized to Euclidean spaces.

The following theorem has been proved in \cite{Ambrosio-Gigli-Savare08} (Corollary 2.4.11):
\begin{theorem}[Existence of gradient flows in the EDE sense]\label{thm:edeex}
Let $(\X,\sfd_\X)$ be a compact  metric space and let
$E:\X\to\R\cup\{+\infty\}$ be a $K$-geodesically convex and lower semicontinuous
functional. Then every $y_0\in D(E)$ is the starting point of a
gradient flow in the EDE sense of $E$.
\end{theorem}

It is important to stress the fact that in general gradient flows in
the EDE sense are \emph{not} unique. A simple example is $\X:=\R^2$
endowed with the $L^\infty$ norm, and $E$ defined by $E(x,y):=x$. It
is immediate to see that $E$ is 0-convex and that for any point
$(x_0,y_0)$ there exist uncountably many gradient flows in the EDE
starting from it, for instance all curves $(x_0-t,y(t))$ with
$|y'(t)|\leq 1$ and $y(0)=y_0$.

\subsubsection{Evolution Variational Inequality}

To see where the EVI notion comes from, notice that for a $K$-convex and smooth function $f$ on $\R^d$
it holds $\x_t'=-\nabla f(\x)$ for any $t\geq 0$ if and only if
\begin{equation}
\label{eq:evird} \frac{\d}{\d t}\frac{|\x_t-z|^2}2+\frac
K2|\x_t-z|^2+f(\x_t)\leq f(z),\qquad\forall z\in\R^d,\ \forall t\geq
0.
\end{equation}
This equivalence is true because $K$-convexity ensures that
$v=-\nabla f(\x)$ if and only
\[
\la v,\x-z\ra+\frac K2|\x-z|^2+f(\x)\leq f(z),\qquad\forall z\in\R^d.
\]
Inequality \eqref{eq:evird} can be written in a metric context in several ways, which we collect
in the following statement (we omit the easy proof).

\begin{proposition}[Evolution Variational Inequality: equivalent statements]\label{prop:equivevi}
Let $(\X,\sfd_\X)$ be a complete and separable metric space, $E:\X\to (-\infty,\infty]$ a lower semicontinuous
functional. Then the following properties are equivalent.
\begin{itemize}
\item[(i)] For any $z\in E$ it holds
\[
\frac{\d}{\d t}\frac{\sfd_\X^2(y_t,z)}2+\frac
K2\sfd_\X^2(y_t,z)+E(y_t)\leq E(z),\qquad\text{for a.e.~$t\in
(0,\infty)$.}
\]
\item[(ii)] For any $z\in E$ it holds
\[
\frac{\sfd_\X^2(y_s,z)-\sfd_\X^2(y_t,z)}2+\frac
K2\int_t^s\sfd_\X^2(y_r,z)\,\d r+\int_t^s E(y_r)\,\d r\leq
(s-t)E(z),\qquad\forall 0<t<s<\infty.
\]
\item[(iii)] There exists a set $A\subset D(E)$ dense in energy (i.e., for any $z\in D(E)$ there
exists $(z_n)\subset A$ converging to $z$ such that $E(z_n)\to
E(z)$) such that for any $z\in A$ it holds
\[
\lims_{h\downarrow0}\frac{\sfd_\X^2(y_{t+h},z)-\sfd_\X^2(y_t,z)}2
+\frac K2\sfd_\X^2(y_t,z)+E(y_t)\leq E(z),\qquad \forall t\in
(0,\infty).
\]
\end{itemize}
\end{proposition}

\begin{definition}[Evolution Variational Inequality definition of gradient flow]
We say that a curve $(y_t)$ is a gradient flow of $E$ in the $EVI$
sense relative to $K\in\R$ (in short, ${\rm EVI}_K$-gradient flow),
if any of the above equivalent properties are true. We say that
$y_t$ starts from $y_0$ if $y_t\to y_0$ as $t\downarrow 0$.
\end{definition}

This definition of gradient flow is stronger than the one discussed
in the previous section, because of the following result proved by
the third author  in \cite{Savare10} (see also Proposition~3.6 of
\cite{Ambrosio-Gigli11}), which we state without proof.

\begin{proposition}[EVI implies EDE]\label{prop:giuseppe}
Let $(\X,\sfd_\X)$ be a complete and separable metric
 space, $K\in\R$, $E:\X\to (-\infty,,\infty]$ a lower
semicontinuous functional and $y_t:(0,\infty)\to D(E)$ a locally
absolutely continuous curve. Assume that $y_t$ is an ${\rm
EVI}_K$-gradient flow for $E$. Then \eqref{eq:defede} holds for any
$0<t<s$.
\end{proposition}

\begin{remark}[Contractivity]\label{re:contr}{\rm
It can be proved that if $(y_t)$ and $(z_t)$ are gradient flows in
the ${\rm EVI}_K$ sense of the l.s.c. functional $E$, then
\[
\sfd_\X(y_t,z_t)\leq e^{-Kt}\sfd_\X(y_0,z_0),\qquad\forall t\geq 0.
\]
In particular, gradient flows in the EVI sense are unique. This
contractivity property, used in conjunction with $(ii)$ of
Proposition~\ref{prop:equivevi}, guarantees that if existence of
gradient flows in the EVI sense is known for initial data lying in
some subset $S\subset \X$, then it is also known for initial data in
the closure $\overline S$ of $S$.}\fr\end{remark}

We also point out the following geometric consequence of the EVI,
proven in \cite{Daneri-Savare08}.

\begin{proposition}\label{prop:dansav}
Let $E:\X\to (-\infty,\infty]$ be a lower semicontinuous functional on a
complete   space $(\X,d_\X)$. Assume that every 
$y_0\in
D(E) $
is the starting point of an ${\rm EVI}_K$-gradient flow of $E$. Then
$E$ is $K$-convex along \emph{all} geodesics contained in $\overline
{D(E)}$.
\end{proposition}

As we already said, gradient flows in the EVI sense do not
necessarily exist, and their existence depends on the properties of
the distance $\sfd_\X$. For instance, it is not hard to see that if
we endow $\R^2$ with the $L^\infty$ norm and consider the functional
$E(x,y):=x$, then there re is no gradient flow in the ${\rm
EVI}_K$-sense, regardless of the constant $K$.

\section{Hopf-Lax formula and Hamilton-Jacobi equation}\label{sec:hopflax}

Aim of this subsection is to study the properties of the Hopf-Lax formula in a metric setting and
its relations with the Hamilton-Jacobi equation. Here we assume that $(X,\sfd)$ is a compact metric space.
Notice that there is no reference measure $\mm$ in the discussion.

Let $f:X\to\R$ be a Lipschitz function. For $t>0$ define
\[
F(t,x,y):=f(y)+\frac{\sfd^2(x,y)}{2t},
\]
and the function $Q_tf:X\to\R$ by
\[
Q_tf(x):=\inf_{y\in X}F(t,x,y)=\min_{y\in X}F(t,x,y).
\]
Also, we introduce the functions $D^+,\,D^-:X\times(0,\infty)\to\R$
as
\begin{equation}\label{eq:defdpm}
\begin{split}
D^+(x,t)&:=\max\, \sfd(x,y),\\
D^-(x,t)&:=\min\, \sfd(x,y),\\
\end{split}
\end{equation}
where, in both cases, the $y$'s vary among all minima of
$F(t,x,\cdot)$. We also set $Q_0f=f$ and $D^\pm(x,0)=0$. Thanks to the continuity of $F$ and the compactness of $X$, it is easy to check
that the map $[0,\infty)\times X\ni(t,x)\mapsto Q_tf(x)$ is
continuous. Furthermore, the fact that $f$ is Lipschitz easily
yields
\begin{equation}
\label{eq:boundD}
D^-(x,t)\leq D^+(x,t)\leq 2t\Lip(f),
\end{equation}
and from the fact that the functions $\{\sfd^2(\cdot,y)\}_{y\in Y}$
are uniformly Lipschitz (because $(X,\sfd)$ is bounded) we get that
$Q_tf$ is Lipschitz for any $t> 0$.

\begin{proposition}[Monotonicity of $D^\pm$]\label{prop:dmon}
For all $x\in X$ it holds
\begin{equation}\label{eq:basic_mono}
D^+(x,t)\leq D^-(x,s),\qquad 0\leq t< s.
\end{equation}
As a consequence, $D^+(x,\cdot)$ and $D^-(x,\cdot)$ are both
nondecreasing, and they coincide with at most countably many
exceptions in $[0,\infty)$.
\end{proposition}
\begin{proof}
Fix $x\in X$. For $t=0$ there is nothing to prove. Now pick
$0<t<s$ and choose $x_t$ and $x_s$ minimizers of $F(t,x,\cdot)$ and $F(s,x,\cdot)$
respectively, such that $\sfd(x,x_t)=D^+(x,t)$ and $
\sfd(x,x_s)=D^-(x,s)$. The minimality of $x_t,x_s$ gives
\[
\begin{split}
f(x_t)+\frac{\sfd^2(x_t,x)}{2t}&\leq f(x_s)+\frac{\sfd^2(x_s,x)}{2t}\\
f(x_s)+\frac{\sfd^2(x_s,x)}{2s}&\leq f(x_t)+\frac{\sfd^2(x_t,x)}{2s}.
\end{split}
\]
Adding up and using the fact that $\tfrac1t\geq\tfrac 1s$ we deduce
\[
D^+(x,t)=\sfd(x_t,x)\leq \sfd(x_s,x)= D^-(x,s),
\]
which is \eqref{eq:basic_mono}.

Combining this with the inequality $D^-\leq D^+$ we
immediately obtain that both functions are nonincreasing. At a point
of right continuity of $D^-(x,\cdot)$ we get
$$
D^+(x,t)\leq\inf_{s>t}D^-(x,s)=D^-(x,t).
$$
This implies that the two functions coincide out of a countable set.
\end{proof}

Next, we examine the semicontinuity properties of $D^\pm$. These
properties imply that points $(x,t)$ where the equality
$D^+(x,t)=D^-(x,t)$ occurs are continuity points for both $D^+$ and
$D^-$.
\begin{proposition}[Semicontinuity of $D^\pm$]
The map $D^+$ is upper semicontinuous and the map $D^-$ is lower
semicontinuous in $X\times (0,\infty)$.
\end{proposition}
\begin{proof}
We prove lower semicontinuity of $D^-$, the proof of upper
semicontinuity of $D^+$ being similar. Let $(x_i,t_i)$ be any
sequence converging to $(x,t)$ and, for every $i$, let $(y_i)$ be
a minimum of $F(t_i,x_i,\cdot)$ for which
$\sfd(y_i,x_i)=D^-(x_i,t_i)$. For all $i$ we have
\[
f(y_i)+\frac{\sfd^2(y_i,x_i)}{2t_i}=Q_{t_i}f(x_i),
\]
Moreover, the continuity of $(x,t)\mapsto Q_tf(x)$ gives that $\lim_iQ_{t_i}f(x_i)=Q_tf(x)$, thus
\[
\lim_{i\to\infty} f(y_i)+\frac{\sfd^2(y_i,x)}{2t}=Q_tf(x).
\]
This means that $(y_i)$ is a minimizing sequence for $F(t,x,\cdot)$.
Since $(X,\sfd)$ is compact, possibly passing to a subsequence, not
relabeled, we may assume that $(y_i)$ converges to $y$ as
$i\to\infty$. Therefore
\[
D^-(x,t)\leq \sfd(x,y)=\lim_{i\to\infty}\sfd(x,y_i)=\lim_{i\to\infty}D^-(x_i,t_i).
\]
\end{proof}

\begin{proposition}[Time derivative of
$Q_tf$]\label{prop:timederivative} The map $t\mapsto Q_tf$ is
Lipschitz from $[0,\infty)$ to $C(X)$  and, for all
$x\in X$, it satisfies
\begin{equation}\label{eq:Dini1}
\frac{\d}{\d t}Q_tf(x)=-\frac{[D^{\pm}(x,t)]^2}{2t^2},
\end{equation}
for any $t>0$ with at most countably many exceptions.
\end{proposition}
\begin{proof}
Let $t<s$ and $x_t$, $x_s$ be minima of $F(t,x,\cdot)$
and $F(s,x,\cdot)$. We have
\[
\begin{split}
Q_sf(x)-Q_tf(x)&\leq F(s,x,x_t)-F(t,x,x_t)=\frac{\sfd^2(x,x_t)}{2}\frac{t-s}{ts},\\
Q_sf(x)-Q_tf(x)&\geq F(s,x,x_s)-F(t,x,x_s)=\frac{\sfd^2(x,x_s)}{2}\frac{t-s}{ts},
\end{split}
\]
which gives that $t\mapsto Q_tf(x)$ is Lipschitz in $(\eps,+\infty)$
for any $\eps>0$ and $x\in X$. Also, dividing by $(s-t)$ and taking
Proposition~\ref{prop:dmon} into account, we get \eqref{eq:Dini1}.
Now notice that from  \eqref{eq:boundD} we get that $|\frac{\d}{\d
t}Q_tf(x)|\leq 2\Lip^2(f)$ for any $x$ and a.e.~$t$, which, together
with the pointwise convergence of $Q_tf$ to $f$ as $t\downarrow 0$,
yields that $t\mapsto Q_tf\in C(X)$ is Lipschitz in $[0,\infty)$.
\end{proof}

\begin{proposition}[Bound on the local Lipschitz constant of $Q_tf$]\label{prop:slopesqt}
For $(x,t)\in X\times (0,\infty)$ it holds:
\begin{equation}
\label{eq:hjbss}
|D Q_tf|(x)\leq \frac{D^+(x,t)}t.
\end{equation}
\end{proposition}
\begin{proof}
Fix $x\in X$ and $t\in (0,\infty)$, pick a sequence $(x_i)$
converging to $x$ and a corresponding sequence $(y_i)$ of minimizers
for $F(t,x_i,\cdot)$ and similarly a minimizer $y$ of $F(t,x,\cdot)$. We start proving  that
\[
\lims_{i\to\infty}\frac{Q_tf(x)-Q_tf(x_i)}{d(x,x_i)}\leq
\frac{D^+(x,t)}{t}.
\]
Since it holds
\[
\begin{split}
Q_tf(x)-Q_tf(x_i)&\leq F(t,x,y_i)-F(t,x_i,y_i)\leq  f(y_i)+\frac{\sfd^2(x,y_i)}{2t}-f(y_i)-\frac{\sfd^2(x_i,y_i)}{2t}\\
&\leq \frac{\sfd(x,x_i)}{2t}\big(\sfd(x,y_i)+\sfd(x_i,y_i)\big)\leq \frac{\sfd(x,x_i)}{2t}\big(\sfd(x,x_i)+2D^+(x_i,t)\big),
\end{split}
\]
dividing by $d(x,x_i)$, letting $i\to\infty$ and using the upper semicontinuity of $D^+$ we get the claim. To conclude, we need to show that
\[
\lims_{i\to\infty}\frac{Q_tf(x_i)-Q_tf(x)}{d(x,x_i)}\leq
\frac{D^+(x,t)}{t}.
\]
This follows along similar lines starting from the inequality
$$Q_tf(x_i)-Q_tf(x)\leq F(t,x_i,y)-F(t,x,y_i).$$
\end{proof}

\begin{theorem}[Subsolution of HJ]\label{thm:subsol}
For every $x\in X$ it holds
\begin{equation}\label{eq:hjbsus}
\frac{\d}{\d t}Q_tf(x)+\frac12|D Q_tf|^2(x)\leq 0
\end{equation}
with at most countably many exceptions in $(0,\infty)$.
\end{theorem}
\begin{proof}
The claim is a direct consequence of
Proposition~\ref{prop:timederivative} and
Proposition~\ref{prop:slopesqt}.
\end{proof}

We just proved that in an arbitrary metric space the Hopf-Lax
formula produces subsolutions of the Hamilton-Jacobi equation. Our
aim now is to prove that if $(X,\sfd)$ is a geodesic space, then the
same formula provides also supersolutions.

\begin{theorem}[Supersolution of HJ]\label{thm:supersol}
Assume that $(X,\sfd)$ is a geodesic space. Then equality holds in
\eqref{eq:hjbss}. In particular, for all $x\in X$ it holds
\[
\frac{\d}{\d t}Q_tf(x)+\frac 12 |D Q_tf|^2(x)=0,
\]
with at most countably many exceptions in $(0,\infty)$.
\end{theorem}
\begin{proof}
Let $y$ be a minimum of $F(t,x,\cdot)$ such that $\sfd(x,y)=D^+(x,t)$. Let
$\gamma:[0,1]\to X$ be a constant speed geodesic connecting $x$ to
$y$.
We have
\[
\begin{split}
 Q_tf(x)-Q_tf(\gamma_s)&\geq f(y)+\frac{\sfd^2(x,y)}{2t}-f(y)-\frac{\sfd^2(\gamma_s,y_i)}{2t}\\
&=\frac{\sfd^2(x,y)-\sfd^2(\gamma_s,y)}{2t}=\frac{\big(D^+(x,t)\big)^2(2s-s^2)}{2t}.
\end{split}
\]
Therefore we obtain
\[
\lims_{s\downarrow 0}\frac{Q_tf(x)-Q_tf(\gamma_s)}{\sfd(x,\gamma_s)}
=\lims_{s\downarrow 0}\frac{Q_tf(x)-Q_tf(\gamma_s)}{sD^+(x,t)}\geq \frac{D^+(x,t)}t
\]
Since $s\mapsto\gamma_s$ is a particular family converging to $x$ we
deduce
\[
|D^-Q_tf|(x)\geq \frac{D^+(x,t)}{t}.
\]
Taking into account Proposition~\ref{prop:timederivative} and
Proposition~\ref{prop:slopesqt} we conclude.
\end{proof}

\section{Weak definitions of gradient}\label{se:Sobolev}

In this section we introduce two weak notions of `norm of the differential', one
inspired by Cheeger's seminal paper \cite{Cheeger00}, that we call
minimal relaxed slope and denote by $\relgrad f$, and one inspired by the papers of Koskela-MacManus
\cite{Koskela-MacManus98} and of Shanmugalingam
\cite{Shanmugalingam00}, 
that we call minimal weak upper gradient and denote by $\weakgrad f$. Notice that, as for the slopes, the objects that we are going to define are naturally in duality with the distance, thus are cotangent notion: that's why we use the `$D$' instead of the `$\nabla$' in the notation. Still, we will continue speaking of upper gradients and their weak counterparts to be aligned with the convention used in the literature (see \cite{Gigli12bis} for a broader discussion on this distinction between tangent and cotangent objects and its effects on calculus).

We compare our concepts with those of the original papers in
Subsection~\ref{sec:comparegradients}, where we show that all these approaches a posteriori coincide.
As usual, we will adopt the simplifying assumption that
$(X,\sfd,\mm)$ is compact and normalized metric measure space, i.e.\ 
$(X,\sfd)$ is compact and $\mm\in \prob X$.

\subsection{The ``vertical'' approach: minimal relaxed slope}\label{se:rellip}

\begin{definition}[Relaxed slopes]\label{def:genuppergrad} We say that $G\in L^2(X,\mm)$ is a relaxed
slope of $f\in L^2(X,\mm)$ if there exist $\tilde{G}\in L^2(X,\mm)$ and Lipschitz
functions $f_n: X\to\R$ such that:
\begin{itemize}
\item[(a)] $f_n\to f$ in $L^2(X,\mm)$ and $|D f_n|$ weakly converges to
$\tilde{G}$ in $L^2(X,\mm)$;
\item[(b)] $\tilde{G}\leq G$ $\mm$-a.e.~in $X$.
\end{itemize}
We say that $G$ is the minimal relaxed slope of $f$ if its
$L^2(X,\mm)$ norm is minimal among relaxed slopes. We shall denote
by $\relgrad f$ the minimal relaxed slope.
\end{definition}

Using Mazur's lemma and \eqref{eq:subadd} 
(see Proposition \ref{le:strongappr})
it is possible to show
that an equivalent characterization of relaxed slopes can be given
by modifying (a) as follows: $\tilde{G}$ is the \emph{strong} limit
in $L^2(X,\mm)$ of $G_n\geq|D f_n|$. The definition of relaxed
slope we gave is useful to show existence of relaxed slopes (as soon
as an approximating sequence $(f_n)$ with $|D f_n|$ bounded in
$L^2(X,\mm)$ exists) while the equivalent characterization is useful
to perform diagonal arguments and to show that the class of relaxed
slopes is a convex closed set. Therefore the definition of $\relgrad
f$ is well posed.

\begin{lemma}[Locality]\label{le:local}
Let $G_1,\,G_2$ be relaxed slopes of $f$. Then $\min\{G_1,G_2\}$ is
a relaxed slope as well. In particular, for any relaxed slope $G$ it
holds
\[
\relgrad  f\leq G\qquad\text{$\mm$-a.e.~in $X$.}
\]
\end{lemma}
\begin{proof}
It is sufficient to prove that if $B\subset X$ is a Borel set, then
$\nchi_BG_1+\nchi_{X\setminus B}G_2$ is a relaxed slope of $f$. By
approximation, taking into account the closure of the class of
relaxed slopes, we can assume with no loss of generality that $B$ is
an open set. We fix $r>0$ and a Lipschitz function $\phi_r:X\to
[0,1]$ equal to $0$ on $X\setminus B_r$ and equal to $1$ on
$B_{2r}$, where the open sets $B_s\subset B$ are defined by
$$
B_s:=\left\{x\in X:\ {\rm dist}(x,X\setminus B)> s\right\}\subset B.
$$
Let now $f_{n,i}$, $i=1,\,2$, be Lipschitz and $L^2$ functions
converging to $f$ in $L^2(X,\mm)$ as $n\to\infty$, with $|D
f_{n,i}|$ weakly convergent to $G_i$ and set $f_n:=\phi_r
f_{n,1}+(1-\phi_r)f_{n,2}$. Then, $|D f_n|=|D f_{n,1}|$ on
$B_{2r}$ and $|D f_n|=|D f_{n,2}|$ on
$X\setminus\overline{B_r}$; in $\overline{B_r}\setminus B_{2r}$, by
applying \eqref{eq:subadd} and \eqref{eq:leibn}, we can estimate
$$
|D f_n|\leq |D f_{n,2}|+{\rm
Lip}(\phi_r)|f_{n,1}-f_{n,2}|+ \phi_r\bigl(|D f_{n,1}|+|D
f_{n,2}|\bigr).
$$
Since $\overline{B_r}\subset B$, by taking weak limits of a
subsequence, it follows that
$$
\nchi_{B_{2r}}G_1+\nchi_{X\setminus\overline{B_r}}G_2+\nchi_{B\setminus
B_{2r}}(G_1+2G_2)
$$
is a relaxed slope of $f$. Letting $r\downarrow 0$ gives that
$\nchi_BG_1+\nchi_{X\setminus B}G_2$ is a relaxed slope as well.

For the second part of the statement argue by contradiction: let $G$
be a relaxed slope of $f$ and assume that $B=\{G<\relgrad f\}$ is
such that $\mm(B)>0$. Consider the relaxed slope $G\nchi_B+\relgrad
f\nchi_{X\setminus B}$: its $L^2$ norm is strictly less than the
$L^2$ norm of $\relgrad f$, which is a contradiction.
\end{proof}

A trivial consequence of the definition and of the locality principle we just proved is that if
$f:X\to\R$ is Lipschitz it holds:
\begin{equation}
\label{eq:facile} \relgrad  f\leq |D f|\qquad\text{$\mm$-a.e.
in $X$.}
\end{equation}

We also remark that it is possible to obtain the minimal relaxed
slope as strong limit in $L^2$ of slopes of Lipschitz functions, and
not only weak, as shown in the next proposition.

\begin{proposition}[Strong approximation]\label{le:strongappr}
If $f\in L^2(X,\mm)$ has a relaxed slope, there exist Lipschitz
functions $f_n$ convergent to $f$ in $L^2(X,\mm)$ with $|D
f_n|$ convergent to $\relgrad  f$ in $L^2(X,\mm)$.
\end{proposition}
\begin{proof}
If $g_i\to f$ in $L^2$ and $|D g_i|$ weakly converges to
$\relgrad  f$ in $L^2$, by Mazur's lemma we can find a sequence of
convex combinations of $|D g_i|$ strongly convergent to
$\relgrad  f$ in $L^2$; the corresponding convex combinations of
$g_i$, that we shall denote by $f_n$, still converge in $L^2$ to $f$
and $|D f_n|$ is dominated by the convex combinations of
$|D g_i|$. It follows that
$$
\lims_{n\to\infty}\int_X|D f_n|^2\,\d\mm\leq
\lims_{i\to\infty}\int_X|D g_i|^2\,\d\mm=\int_X\relgrad f^2\,\d\mm.
$$
This implies at once that $|D f_n|$ weakly converges to
$\relgrad  f$ (because any limit point in the weak topology is a
relaxed slope with minimal norm) and that the convergence is strong.
\end{proof}

\begin{theorem} \label{thm:cheeger} The Cheeger energy functional
\begin{equation}\label{def:Cheeger}
\C(f):=\frac{1}{2}\int_X \relgrad f^2 \,\d\mm,
\end{equation}
set to $+\infty$ if $f$ has no relaxed slope, is convex and lower
semicontinuous in $L^2(X,\mm)$.
\end{theorem}
\begin{proof} A simple byproduct of condition \eqref{eq:subadd} is that $\alpha F+\beta G$ is a
relaxed slope of $\alpha f+\beta g$ whenever $\alpha,\,\beta$ are
nonnegative constants and $F,\,G$ are relaxed slopes of $f,\,g$
respectively. Taking $F=\relgrad f$ and $G=\relgrad g$ yields the
convexity of $\C$, while lower semicontinuity follows by a simple
diagonal argument based on the strong approximation property stated
in Proposition~\ref{le:strongappr}.
\end{proof}

\begin{proposition}[Chain rule]\label{prop:chain}
If $f\in L^2(X,\mm)$ has a relaxed slope and $\phi:X\to\R$ is
Lipschitz and $C^1$, then $\relgrad {\phi(f)}= |\phi'(f)|\relgrad
f$ $\mm$-a.e.~in $X$.
\end{proposition}
\begin{proof}
We trivially have $|D\phi(f)|\leq|\phi'(f)||D f|$. If we
apply this inequality to the ``optimal'' approximating sequence of
Lipschitz functions given by Proposition~\ref{le:strongappr} we get
that $|\phi'(f)|\relgrad{f}$ is a relaxed slope of $\phi(f)$, so
that $\relgrad {\phi(f)}\leq |\phi'(f)|\relgrad f$ $\mm$-a.e.~in
$X$. Applying twice this inequality with $\phi(r):=-r$ we get $\relgrad f\leq\relgrad{(-f)}\leq \relgrad f$ and thus $\relgrad f=\relgrad{(-f)}$ 
$\mm$-a.e. in $X$. 

Up to a simple rescaling, we can assume $|\phi'|\leq 1$.  Let $\psi_1(z):=z-\phi(z)$, notice that $\psi_1'\geq 0$ and thus 
$\mm$-a.e. on $f^{-1}(\{\phi'\geq 0\})$ it holds
$$
\relgrad{f}\leq\relgrad{(\phi(f))}+\relgrad{(\psi_1(f))}\leq
\phi'(f)\relgrad{f}+\psi_1'(f)\relgrad{f}=\relgrad f,
$$
hence all the inequalities must be equalities, which forces $\relgrad{(\phi(f))}=\phi'(f)\relgrad f$ $\mm$-a.e. on $f^{-1}(\{\phi'\geq 0\})$.
Similarly, let $\psi_2(z)=-z-\phi(z)$ and notice that $\psi_2'\leq 0$, so that $\mm$-a.e. on $f^{-1}(\{\phi'\leq 0\})$ it holds
\[
\relgrad f=\relgrad{(-f)}\leq \relgrad{(\phi(f))}+\relgrad{(\psi_2(f))}\leq
-\phi'(f)\relgrad{f}-\psi_2'(f)\relgrad{f}=\relgrad f.
\] 
As before we can conclude that $\relgrad{(\phi(f))}=-\phi'(f)\relgrad f$ $\mm$-a.e. on $f^{-1}(\{\phi'\leq 0\})$.
\end{proof}

Still by approximation, it is not difficult to show that $\phi(f)$ has a
relaxed slope if $\phi$ is Lipschitz, and that
$\relgrad {\phi(f)}=|\phi'(f)|\relgrad f$ $\mm$-a.e.~in $X$. In this
case $\phi'(f)$ is undefined at points $x$ such that $\phi$ is not
differentiable at $f(x)$, on the other hand the formula still makes
sense because $\relgrad{f}=0$ $\mm$-a.e.~on $f^{-1}(N)$ for any
Lebesgue negligible set $N\subset\R$. Particularly useful is the
case when $\phi$ is a truncation function, for instance
$\phi(z)=\min\{z,M\}$. In this case
$$
\relgrad{\min\{f,M\}}=
\begin{cases}
\relgrad{f}&\text{if $f(x)<M$}
\\
0&\text{if $f(x)\geq M$.}
\end{cases}
$$
Analogous formulas hold for truncations from below.

\subsubsection{Laplacian: definition and basic properties}\label{se:laplacian}

Since the domain of $\C$ is dense in $L^2(X,\mm)$ (it includes
Lipschitz functions), the Hilbertian theory of gradient flows (see
for instance \cite{Brezis73}, \cite{Ambrosio-Gigli-Savare08}) can be
applied to Cheeger's functional \eqref{def:Cheeger} to provide, for
all $f_0\in L^2(X,\mm)$, a locally Lipschitz continuous map
$t\mapsto f_t$ from $(0,\infty)$ to $L^2(X,\mm)$, with $f_t\to f_0$
as $t\downarrow 0$, whose derivative satisfies
\begin{equation}\label{eq:ODE}
  \frac{\d}{\d t}f_t\in -\partial\C(f_t)\qquad\text{for a.e.~$t$.}
\end{equation}
Here $\partial\C(g)$ denotes the subdifferential of $\C$ at $g\in
D(\C)$ in the sense of convex analysis, i.e.
$$
\partial\C(g):=\left\{\xi\in L^2(X,\mm):\
\C(f)\geq\C(g)+\int_X\xi(f-g)\,\d\mm\,\,\,\forall f\in
L^2(X,\mm)\right\}.
$$
Another important regularizing effect of gradient flows of convex
l.s.c. functionals lies in the fact that 
for every $t>0$ (the opposite of) the right derivative
$-\tfrac{\d}{\d t_+} f_t=\lim_{h\downarrow0}\frac 1h(f_t-f_{t+h})$
exists and it is actually the element with minimal
$L^2(X,\mm)$ norm in $\partial^-\C(f_t)$. This motivates the next
definition:
\begin{definition}[Laplacian]\label{def:delta}
The Laplacian $\Delta f$ of $f\in L^2(X,\mm)$ is defined for those
$f$ such that $\partial\C(f)\neq\emptyset$. For those $f$, $-\Delta
f$ is the element of minimal $L^2(X,\mm)$ norm in $\partial\C(f)$.
The domain of $\Delta$ is defined as $D(\Delta)$.
\end{definition}
\begin{remark}[Potential lack of linearity]\label{re:laplnonlin}{\rm
It should be observed that in general the Laplacian - as we just
defined it - is \emph{not} a linear operator: the potential lack of
linearity is strictly related to the fact that potentially the space
$W^{1,2}(X,\sfd,\mm)$ is not Hilbert, because
$f\mapsto\int\relgrad{f}^2\,\d\mm$ need not be quadratic. For
instance if $X=\R^2$, $\mm$ is the Lebesgue measure and $\sfd$ is
the distance induced by the $L^\infty$ norm, then it is easily seen
that
$$
\relgrad{f}^2=\biggl(\biggl|\frac{\partial f}{\partial x}\biggr|+
\biggl|\frac{\partial f}{\partial y}\biggr|\biggr)^2.
$$
Even though the Laplacian is not linear, the trivial implication
\[
v\in\partial^- \C(f)\qquad\Rightarrow\qquad \lambda v\in\partial^- \C(\lambda f),\quad\forall \lambda\in\R,
\]
ensures that the Laplacian (and so the gradient flow of $\C$) is 1-homogenous.
}\fr
\end{remark}

We can now write
$$
\frac{\d}{\d t}f_t=\Delta f_t
$$
for gradient flows $f_t$ of $\C$, the derivative being understood in
$L^2(X,\mm)$, in accordance with the classical case. The classical
Hilbertian theory of gradient flows also ensures that
\begin{equation}
\label{eq:perdopo} \lim_{t\to\infty}\C(f_t)=0\qquad\text{and}\qquad
\frac {\d}{\d t}\C(f_t)=-\|\Delta f_t\|^2_{L^2(X,\smm)},\quad
\text{for a.e.~$t\in (0,\infty)$.}
\end{equation}

\begin{proposition}[Integration by parts]
\label{prop:deltaineq} For all $f\in D(\Delta)$, $g\in D(\C)$ it
holds
\begin{equation}
\label{eq:delta1} \left|\int_X g\Delta f\,\d\mm\right|\leq \int_X \relgrad g\relgrad f\,\d\mm.
\end{equation}
Also, let $f\in D(\Delta)$ and $\phi\in C^1(\R)$ with bounded
derivative on an interval containing the image of $f$. Then
\begin{equation}
\label{eq:delta2} \int_X \phi(f)\Delta f\,\d\mm= -\int_X \relgrad f^2\phi'(f)\,\d\mm.
\end{equation}
\end{proposition}
\begin{proof}
Since $-\Delta f\in\partial^-\C(f)$ it holds
\[
\C(f)-\int_X \eps g\Delta f\,\d\mm\leq \C(f+\eps g),\qquad\forall g\in
L^2(X,\mm),\,\,\eps\in\R.
\]
For $\eps>0$, $\relgrad f+\eps \relgrad g$ is a relaxed slope of
$f+\eps g$ (possibly not minimal). Thus it holds $2\C(f+\eps
g)\leq\int_X(\relgrad f+\eps\relgrad g)^2\,\d\mm$ and therefore
\[
-\int_X\eps g\Delta f\,\d\mm\leq \frac12\int_X(\relgrad f+\eps\relgrad g)^2-\relgrad f^2\,\d\mm=
\eps\int_X\relgrad f\relgrad g\,\d\mm+o(\eps).
\]
Dividing by $\eps$, letting $\eps\downarrow 0$ and then repeating the argument with $-g$ in place of $g$ we get
\eqref{eq:delta1}.

For the second part we recall that, by the chain rule, $\relgrad
{(f+\eps \phi(f))}=(1+\eps \phi'(f))\relgrad f$ for $|\eps|$ small
enough. Hence
\[
\C(f+\eps \phi(f))-\C(f)= \frac{1}{2}\int_X\relgrad f^2\bigl((1+\eps
\phi'(f))^2-1\bigr)\,\d\mm=\eps\int_X\relgrad f^2
\phi'(f)\,\d\mm+o(\eps),
\]
which implies that for any $v\in \partial^-\C(f)$ it holds $\int_Xv
\phi(f)\,\d\mm=\int_X\relgrad f^2\phi'(f)\,\d\mm$, and gives the thesis
with $v=-\Delta f$.
\end{proof}

\begin{proposition}[Some properties of the gradient flow of $\C$]\label{prop:basecal}
Let $f_0\in L^2(X,\mm)$ and let $(f_t)$ be the gradient flow of $\C$
starting from $f_0$. Then the following properties hold.\\*
\noindent\underline{Mass preservation.} $\int f_t\,\d\mm=\int
f_0\,\d\mm$ for any $t\geq 0$.\\* \noindent\underline{Maximum
principle.} If $f_0\leq C$ (resp. $f_0\geq c$) $\mm$-a.e.~in $X$,
then $f_t\leq C$ (resp $f_t\geq c$) $\mm$-a.e.~in $X$ for any $t\geq
0$.\\* \noindent\underline{Entropy dissipation.} Suppose $0<c\leq
f_0\leq C<\infty$ $\mm$-a.e.. Then $t\mapsto\int f_t\log f_t\,\d\mm$
is absolutely continuous in $[0,\infty)$
and it holds
\[
\frac{\d}{\d t}\intX f_t\log
f_t\,\d\mm=-\intX\frac{\relgrad{f_t}^2}{f_t}\,\d\mm,\qquad \text{for
a.e.~$t\in (0,\infty)$.}
\]
\end{proposition}
\begin{proof} \\*\underline{Mass preservation.} Just notice that from  \eqref{eq:delta1} we get
\[
\left|\frac{\d}{\d t}\intX f_t\,\d\mm\right|=\left|\intX
\mathbf{1}\cdot\Delta f_t\,\d\mm\right|\leq\intX\relgrad{\mathbf
1}\relgrad {f_t}\,\d\mm=0,\qquad \text{for a.e.~$t\in (0,\infty)$,}
\]
where $\mathbf 1$ is the function identically equal to 1, which has
minimal relaxed gradient equal to 0.\\* \underline{Maximum
principle.} Fix $f\in L^2(X,\mm)$, $\tau>0$ and, according to the
implicit Euler scheme, let $f^\tau$ be the unique minimizer of
\[
g\qquad\mapsto\qquad \C(g)+\frac{1}{2\tau}\int_X|g-f|^2\,\d\mm.
\]
Assume that $f\leq C$. We claim that in this case $f^\tau\leq C$ as
well. Indeed, if this is not the case we can consider the competitor
$g:=\min\{f^\tau,C\}$ in the above minimization problem. By $(a)$ of
Proposition~\ref{prop:chain} we get $\C(g)\leq\C(f^\tau)$ and the
$L^2$ distance of $f$ and $g$ is strictly smaller than the one of
$f$ and $f^\tau$ as soon as $\mm(\{f^\tau>C\})>0$, which is a
contradiction.

Starting from $f_0$, iterating this
procedure, and using the fact that the implicit Euler scheme
converges as $\tau\downarrow 0$ (see \cite{Brezis73},
\cite{Ambrosio-Gigli-Savare08} for details) to the gradient flow we
get the conclusion.\\
The same arguments applies to uniform bounds from below.\\*
\underline{Entropy dissipation.} The map $z\mapsto z\log z$ is
Lipschitz on $[c,C]$ which, together with the maximum principle and
the fact that $t\mapsto f_t\in L^2(X,\mm)$ is locally absolutely
continuous, yields the claimed absolute continuity statement. Now
notice that we have $\frac{\d}{\d t}\int f_t\log f_t\,\d\mm=\int
(\log f_t+1)\Delta f_t\,\d\mm$ for a.e.~$t$. Since by the maximum
principle $f_t\geq c$ $\mm$-a.e., the function $\log z+1$ is
Lipschitz and $C^1$ on the image of $f_t$ for any $t\geq 0$, thus
from \eqref{eq:delta2} we get the conclusion.
\end{proof}

\subsection{The ``horizontal'' approach: weak upper gradients}

In this subsection, following the approach of
\cite{Ambrosio-Gigli-Savare11,Ambrosio-Gigli-Savare11tris},
we introduce a different notion of ``weak norm of
gradient'' in a compact and normalized metric measure space
$(X,\sfd,\mm)$.
This notion of gradient is Lagrangian in spirit, 
it does not require a relaxation procedure, it
will provide a new estimate of entropy dissipation along the
gradient flow of $\C$, and it will also be useful in the analysis of
the derivative of the entropy along Wasserstein geodesics.

While the definition of minimal relaxed slope was taken from
Cheeger's work \cite{Cheeger00}, the notion we are going to
introduce is inspired by the work of 
Koskela-MacManus
\cite{Koskela-MacManus98} and Shanmugalingam
\cite{Shanmugalingam00}, 
the only difference being that we consider
a different notion of null set of curves.

\subsubsection{Negligible sets of curves and functions Sobolev along a.e.~curve}

Recall that the evaluation maps $\rme_t:C([0,1],X)\to X$ are
defined by $\rme_t(\gamma):=\gamma_t$. We also introduce the
restriction maps ${\rm restr}_t^s: C([0,1],X)\to C([0,1],X)$,
$0\le t\le s\le 1$, given by
\begin{equation}
{\rm restr}_t^s(\gamma)_r:=\gamma_{((1-r)t+rs)},\label{eq:93}
\end{equation}
so that ${\rm restr}_t^s$ restricts the curve $\gamma$ to the
interval $[t,s]$ and then ``stretches'' it on the whole of $[0,1]$.

\begin{definition}[Test plans and negligible sets of curves]\label{def:testplans}
We say that a probability measure $\ppi\in\prob{C([0,1],X)}$ is a
\emph{test plan} if 
it is concentrated on $AC^2([0,1];X)$,
$\iint_0^1|\dot\gamma_t|^2\d t\,\d\ppi<\infty$,
and there exists a constant $C(\ppi)$ such that

\begin{equation}
(\e_t)_\sharp \ppi \leq C(\ppi)\mm\qquad\forevery t\in[0,1].\label{eq:9}
\end{equation}
A Borel set $A\subset AC^2([0,1],X)$ is said \emph{negligible} if
$\ppi(A)=0$ for any test plan $\ppi$. A property which holds for
every $\gamma\in AC^2([0,1],X)$, except possibly a negligible set,
is said to hold for almost every curve.
\end{definition}
\begin{remark}
  \label{re:easy}
  \upshape
  An easy consequence of condition \eqref{eq:9} is that if two
  $\mm$-measurable functions $f,\,g:X\to\R$ coincide up to a
  $\mm$-negligible set and $\mathcal T$ is an at most countable subset
  of $[0,1]$, then the functions
  $f\circ \gamma$ and $g\circ \gamma$ coincide in $\mathcal T$ 
  for almost every curve
  $\gamma$. 

  Moreover, choosing an arbitrary test plan $\ppi$ and applying Fubini's
  Theorem to the product measure $\Leb 1\times \ppi$
  in $(0,1)\times C([0,1];X)$ we also obtain that
  $f\circ\gamma=g\circ\gamma$ $\Leb 1$-a.e.\ in $(0,1)$ for
  $\ppi$-a.e.\ curve $\gamma$; since $\ppi$ is arbitrary, the same
  property holds for almost every curve.
\end{remark}
Coupled with the definition of
negligible set of curves, there are the definitions of weak upper gradient and
 of functions which are Sobolev along a.e.\ curve.
\begin{definition}[Weak upper gradients]
A Borel function
$g:X\to[0,\infty]$ is a weak upper gradient of $f:X\to \R$ if
\begin{equation}
\label{eq:inweak} \left|\int_{\partial\gamma}f\right|\leq
\int_\gamma g\qquad\text{for a.e. $\gamma$.}
\end{equation}
\end{definition}
\begin{definition}[Sobolev functions along a.e. curve]
 A function $f:X\to\R$ is Sobolev along a.e.~curve if for
 a.e.~curve $\gamma$ the function $f\circ\gamma$ coincides a.e.~in $[0,1]$ and in $\{0,1\}$ with an absolutely continuous map
 $f_\gamma:[0,1]\to\R$.
\end{definition}
By Remark \ref{re:easy} applied to $\mathcal T:=\{0,1\}$, \eqref{eq:inweak} does not depend on
the particular representative of $f$ in the class of $\mm$-measurable
function coinciding with $f$ up to a $\mm$-negligible set. 
The same Remark also shows that  the property of being Sobolev along
almost every curve $\gamma$ is independent of the representative in the
class of $\mm$-measurable functions coinciding with $f$ $\mm$-a.e.\
in $X$.

In the following remarks we will make use of this basic calculus lemma:
\begin{lemma}\label{lem:Fibonacci}
{
Let $f:(0,1)\to\R$ Lebesgue measurable,
$q\in [1,\infty {\color{black}]}$, $g\in L^q(0,1)$ nonnegative be satisfying
$$
|f(s)-f(t)|\leq\bigl|\int_s^t g(r)\,\d r\bigr|\qquad\text{for $\Leb{2}$-a.e. $(s,t)\in (0,1)^2$.}
$$
Then $f\in W^{1,q}(0,1)$ and $|f'|\leq g$ a.e. in $(0,1)$.}
\end{lemma}
\begin{proof} 
It is immediate to check that $f\in L^\infty(0,1)$.
Let $N\subset (0,1)^2$ be the $\Leb{2}$-negligible subset where the above 
inequality fails. By Fubini's theorem, also the set $\{(t,h)\in (0,1)^2:\ (t,t+h)\in N\cap (0,1)^2\}$
is $\Leb{2}$-negligible. In particular, by Fubini's theorem, for a.e. $h$ we have
$(t,t+h)\notin N$ for a.e. $t\in (0,1)$. Let $h_i\downarrow 0$ with this property and use
the identities
$$
\int_0^1f(t)\frac{\phi(t+h)-\phi(t)}{h}\,\d t=-\int_0^1\frac{f(t-h)-f(t)}{-h}\phi(t)\,\d t
$$
with $\phi\in C^1_c(0,1)$ and $h=h_i$ sufficiently small to get
$$
\biggl|\int_0^1f(t)\phi'(t)\,\d t\biggr|\leq\int_0^1g(t)|\phi(t)|\,\d t.
$$
It follows that the distributional derivative of $f$ is a signed
measure $\eta$ with finite total variation which satisfies
\begin{displaymath}
  -\int_0^1f\,\phi'\,\d t=\int_0^1 \phi\,\d\eta,\qquad
  \Bigl|\int_0^1 \phi\,\d\eta\Bigr|\le \int_0^1g\,|\phi|\,\d
t\quad\text{for every }\phi\in C^1_c(0,1);
\end{displaymath}
therefore $\eta$ is absolutely continuous with respect to the Lebesgue
measure with $|\eta|\le g\Leb 1$. 
This gives the $W^{1,1}(0,1)$ regularity and, at the same time, the inequality
$|f'|\leq g$ a.e. in $(0,1)$. The case $q>1$ immediately follows
by applying this inequality when $g\in L^q(0,1)$.
\end{proof}
With the aid of this lemma, we can  prove that the existence of a weak
upper gradient $g$ such that $\int_\gamma g<\infty$ for
a.e.\ $\gamma$ (in particular if $g\in L^2(X,\mm)$) implies Sobolev regularity along a.e.\ curve. 
\begin{remark}[Restriction and equivalent formulation]
  \label{re:restr}{\rm
{Notice that if $\ppi$ is a test plan, so is $({\rm
restr}_t^s)_\sharp\ppi$. Hence if $g$ is a weak upper gradient
of $f$ such that $\int_\gamma g<\infty$ for
a.e.\ $\gamma$, then for every $t<s$ in $[0,1]$ it holds
    \[
    |f(\gamma_s)-f(\gamma_t)|\leq \int_t^s g(\gamma_r)|\dot\gamma_r|\,\d
    r \qquad\text{for a.e. $\gamma$.}
    \]
    Let $\ppi$ be a test plan: by Fubini's theorem applied
    to the product measure $\Leb2\times\ppi$ in $(0,1)^2\times
    C([0,1];X)$, it follows that for $\ppi$-a.e. $\gamma$ the function
     $f$ satisfies
    \[
    |f(\gamma_s)-f(\gamma_t)|\leq \Bigl|\int_t^s g(\gamma_r)|\dot\gamma_r|\,\d
    r \Bigr|\qquad\text{for $\Leb{2}$-a.e. $(t,s)\in (0,1)^2$.}
    \]
    An analogous argument shows that 
    \begin{equation}
      \label{eq:2}
      \left\{
    \begin{aligned}
      \textstyle |f(\gamma_s)-f(\gamma_0)|&\textstyle 
      \leq \int_0^s
      g(\gamma_r)|\dot\gamma_r|\,\d r\\
      \textstyle |f(\gamma_1)-f(\gamma_s)|&\textstyle \leq \int_s^1
      g(\gamma_r)|\dot\gamma_r|\,\d r
    \end{aligned}\right.
    \qquad\text{for $\Leb{1}$-a.e. $s\in (0,1)$.}
\end{equation}
 Since $g\circ \gamma|\dot \gamma|\in L^1(0,1)$ for
    $\ppi$-a.e.\ $\gamma$,  
    by Lemma~\ref{lem:Fibonacci} it follows that $f\circ\gamma\in W^{1,1}(0,1)$
    for $\ppi$-a.e. $\gamma$, and
    \begin{equation}\label{eq:pointwisewug}
      \biggl|\frac{\d}{\dt}(f\circ\gamma)\biggr|\leq
      g\circ\gamma|\dot\gamma|\quad\text{a.e. in $(0,1)$, for
        $\ppi$-a.e. $\gamma$.}
    \end{equation}}%
  Since $\ppi$ is arbitrary, we conclude that $f\circ\gamma\in
  W^{1,1}(0,1)$ for a.e.\ $\gamma$, and therefore it admits an
  absolutely continuous representative $f_\gamma$; moreover,
  by \eqref{eq:2}, it is immediate to check that $f(\gamma(t))=f_\gamma(t)$ for $t\in \{0,1\}$ and a.e.\ $\gamma$.
    \fr   }
\end{remark}
\begin{remark}[An approach with a non explicit use of negligible set of curves]{\rm
The previous remark could be used to introduce the notion of weak upper gradients without speaking (explicitly) of Borel sets at all. One can simply say that $g\in L^2(X,\mm)$ is a weak upper gradient of  $f:X\to\R$  provided it holds
\[
\int|f(\gamma_1)-f(\gamma_0)|\,\d\ppi(\gamma)\leq \iint_0^1g(\gamma_s)|\dot\gamma_s|\,\d s\,\d\ppi(\gamma).
\]
(this has been the approach followed in \cite{Gigli12bis}).
\fr}\end{remark}

\begin{proposition}[Locality]\label{prop:locweak}
Let $f:X\to\R$ be Sobolev along almost all absolutely continuous
curves, and let $G_1,\,G_2$ be weak upper gradients of $f$. Then
$\min\{G_1,G_2\}$ is a weak upper gradient of $f$.
\end{proposition}
\begin{proof} It is a direct consequence of \eqref{eq:pointwisewug}.
\end{proof}
\begin{definition}[Minimal weak upper gradient]
  Let $f:X\to\R$ be Sobolev along almost all curves.
  The minimal weak upper gradient $\weakgrad f$ of $f$
  is the weak upper gradient characterized, up to
$\mm$-negligible sets, by the property
\begin{equation}\label{eq:defweakgrad}
  \weakgrad f\leq G\qquad\text{$\mm$-a.e.~in $X$, for every weak upper
    gradient $G$ of $f$.}
\end{equation}
\end{definition}
Uniqueness of the minimal weak upper gradient is obvious. For
existence, we take $\weakgrad f:=\inf_n G_n$, where $G_n$ are weak
upper gradients which provide a minimizing sequence in
$$
\inf\left\{\int_X {\rm tan}^{-1}G\,\d\mm:\ \text{$G$ is a weak upper
gradient of $f$}\right\}.
$$
We immediately see, thanks to Proposition~\ref{prop:locweak}, that
we can assume with no loss of generality that $G_{n+1}\leq G_n$.
Hence, by monotone convergence, the function $\weakgrad f$ is a
weak upper gradient of $f$ and $\int_X {\rm
tan}^{-1}G\,\d\mm$ is minimal at $G=\weakgrad f$. This
minimality, in conjunction with Proposition~\ref{prop:locweak},
gives \eqref{eq:defweakgrad}.

\begin{theorem}[Stability w.r.t. $\mm$-a.e.~convergence]\label{thm:stabweak}
Assume that $f_n$ are $\mm$-measurable, Sobolev along almost
all curves and that $G_n$ are weak upper gradients of $f_n$.
Assume furthermore that $f_n(x)\to f(x)\in\R$ for $\mm$-a.e.~$x\in
X$ and that $(G_n)$ weakly converges to $G$ in $L^2(X,\mm)$. Then $G$ is a weak upper gradient
of $f$.
\end{theorem}
\begin{proof}
Fix a test plan $\ppi$.
By Mazur's theorem we can find convex combinations
$$
H_n:=\sum_{i=N_h+1}^{N_{h+1}}\alpha_iG_i\qquad\text{with
$\alpha_i\geq 0$, $\sum_{i=N_h+1}^{N_{h+1}}\alpha_i=1$,
$N_h\to\infty$}
$$
converging strongly to $G$ in $L^2(X,\mm)$. Denoting by $\tilde f_n$
the corresponding convex combinations of $f_n$, $H_n$ are weak upper
gradients of $\tilde f_n$ and still $\tilde f_n\to f$ $\mm$-a.e.~in
$X$.

Since for every nonnegative Borel function $\varphi:X\to [0,\infty]$ it holds (with $C=C(\ppi)$)
\begin{align}
  \notag\int\Big(\int_{\gamma}\varphi\Big)\,\d\ppi&=
  \int\Big(\int_0^1 \varphi(\gamma_t)|\dot
  \gamma_t|\,\d t\Big)\,\d\ppi
  \le
  \int\Big(\int_0^1\varphi^2(\gamma_t)\,\d
  t\Big)^{1/2}
  \Big(\int_0^1 |\dot
  \gamma_t|^2\,\d t\Big)^{1/2}\,\d\ppi
  \\&\notag
  \le
  \Big(\int_0^1 \int\varphi^2\,\d(\rme_t)_\sharp\ppi\,\d
  t\Big)^{1/2}
  \Big(\iint_0^1|\dot\gamma_t|^2\,\d t\,\d\ppi\Big)^{1/2}
\\ &\le  \Big(C\int\varphi^2\,\d\mm\Big)^{1/2}  \Big(\iint_0^1|\dot\gamma_t|^2\,\d t\,\d\ppi\Big)^{1/2},
\label{eq:21}
 \end{align}
we obtain, for $\bar C:=\sqrt{C} \Big(\iint_0^1|\dot\gamma_t|^2\,\d
t\,\d\ppi\Big)^{1/2}$,
\begin{align*}
\int&\biggl(\int_{\gamma}|H_n-G|+\min\{|\tilde{f}_n-f|,1\}\biggr)\,\d\ppi\leq
\bar C\Big(\|H_n-G\|_{L^2}+
\|\min\{|\tilde{f}_n-f|,1\}\|_{L^2}\Big) \to 0.
\end{align*}
By a diagonal argument we can find a subsequence $n(k)$  such that
$\int_\gamma|H_{n(k)}-G|+\min\{|\tilde{f}_{n(k)}-f|,1\}\to 0$ as
$k\to\infty$ for $\ppi$-a.e.~$\gamma$. Since
$\tilde{f}_n$ converge $\mm$-a.e.~to $f$ and the marginals of $\ppi$
are absolutely continuous w.r.t. $\mm$ we have also that for
$\ppi$-a.e.~$\gamma$ it holds $\tilde{f}_n(\gamma_0)\to f(\gamma_0)$
and $\tilde{f}_n(\gamma_1)\to f(\gamma_1)$.

If we fix a curve $\gamma$ satisfying these convergence properties,
since $(\tilde{f}_{n(k)})_\gamma$ are equi-absolutely continuous
(being their derivatives bounded by
$H_{n(k)}\circ\gamma|\dot\gamma|$) and a further subsequence of
$\tilde{f}_{n(k)}$ converges a.e.~in $[0,1]$ and in $\{0,1\}$ to
$f(\gamma_s)$, we can pass to the limit to obtain an absolutely
continuous function $f_\gamma$ equal to $f(\gamma_s)$ a.e.~in
$[0,1]$ and in $\{0,1\}$ with derivative bounded by
$G(\gamma_s)|\dot\gamma_s|$. Since $\ppi$ is arbitrary
we conclude that $f$ is Sobolev along almost all curves and
that $G$ is a weak upper gradient of $f$.
\end{proof}

\begin{remark}[$\weakgrad f\leq\relgrad f$]\label{re:weakrel}{\rm
An immediate consequence of the previous proposition is that
any $f\in D(\C)$ is Sobolev along a.e.~curve and satisfies
$\weakgrad f\leq\relgrad f$. Indeed, for such $f$ just pick a
sequence of Lipschitz functions converging to $f$ in $L^2(X,\mm)$ such that $|D f_n|\to\relgrad
f$ in $L^2(X,\mm)$ (as in Proposition~\ref{le:strongappr}) and
recall that for Lipschitz functions the local Lipschitz constant is
an upper gradient.}\fr
\end{remark}

\subsubsection{A bound from below on weak gradients}
In this short subsection we show how, using test plans and the very
definition of minimal weak gradients,  it is possible to use
$\weakgrad f$ to bound from below the increments of the relative
entropy. We start with the following result, proved - in a more
general setting - by Lisini in \cite{Lisini07}: it shows how to
associate to a curve 
$\mu\in AC^2([0,1];(\prob X,W_2))$ 
a plan $\ppi\in\prob{C([0,1],X)}$ 
concentrated on $AC^2([0,1];X)$ representing the curve itself (see
also Theorem 8.2.1 of \cite{Ambrosio-Gigli-Savare08} for the
Euclidean case). We will only sketch the proof.
\begin{proposition}[Superposition principle]\label{prop:lisini}
Let $(X,\sfd)$ be a compact space and let 
$\mu\in AC^2([0,1];(\prob
X,W_2))$.
Then there exists
$\ppi\in\prob{C([0,1],X)}$ concentrated on $AC^2([0,1];X)$ such that $(\e_t)_\sharp\ppi=\mu_t$ for
any $t\in[0,1]$ and
$\int|\dot\gamma_t|^2\,\d\ppi(\gamma)=|\dot\mu_t|^2$ for a.e.
$t\in[0,1]$.
\end{proposition}
\begin{proof}
If $\ppi\in C([0,1],X)$ is any plan concentrated on
$AC^2([0,1],X)$ such that $(\e_t)_\sharp\ppi=\mu_t$ for any
$t\in[0,1]$, since
$(\e_t,\e_s)_\sharp\ppi\in\adm(\mu_t,\mu_s)$, for any $t<s$ it holds
\[
W_2^2(\mu_t,\mu_s)\leq\int\sfd^2(\gamma_t,\gamma_s)\,\d\ppi(\gamma)
\leq\int\left(\int_t^s|\dot\gamma_r|\,\d r\right)^2\,\d\ppi(\gamma)
\leq(s-t)\iint_t^s|\dot\gamma_r|^2\,\d r\,\d\ppi(\gamma),
\]
which shows that $|\dot\mu_t|^2\leq
\int|\dot\gamma_t|^2\,\d\ppi(\gamma)$ for a.e.~$t$. Hence, to
conclude it is sufficient to find a plan $\ppi\in
\prob{C([0,1],X)}$, concentrated on $AC^2([0,1],X)$, with
$(\e_t)_\sharp\ppi=\mu_t$ for any $t\in[0,1]$ such that
$\int|\dot\mu_t|^2\,\d t\geq\iint_0^1|\dot\gamma_t|^2\,\d
t\,\d\ppi(\gamma)$.

To build such a $\ppi$ we make the simplifying assumption that
$(X,\sfd)$ is geodesic (the proof for the general case is similar,
but rather than interpolating with piecewise geodesic curves one
uses piecewise constant ones, this leads to some technical
complications that we want to avoid here - see \cite{Lisini07} for
the complete argument). Fix $n\in\N$ and use a gluing argument to
find $\ggamma^n\in\prob{X^{n+1}}$ such that
$(\pi^i,\pi^{i+1})_\sharp\ggamma^n\in\opt(\mu_{\frac
in},\mu_{\frac{i+1}n})$ for $i=0,\ldots,n-1$. By standard measurable
selection arguments, there exists a Borel map $T^n:X^{n+1}\to
C([0,1],X)$ such that $\gamma:=T^n(x_0,\ldots,x_n)$ is a constant
speed geodesic on each of the intervals $[i/n,(i+1)/n]$ and
$\gamma_{i/n}=x_i$, $i=0,\ldots,n$. Define
$\ppi^n:=T^n_\sharp\ggamma^n$. It holds
\begin{equation}
\label{eq:persuperposition} \iint_0^1|\dot\gamma_t|^2\,\d
t\,\d\ppi^n(\gamma)=
\frac1n\int\sum_{i=0}^{n-1}\sfd^2\big(\gamma_{\frac
in},\gamma_{\frac{i+1}n}\big)\,\d\ppi(\gamma)=
\frac1n\sum_{i=0}^{n-1}W_2^2\big(\mu_{\frac
in},\mu_{\frac{i+1}n}\big)\leq\int_0^1|\dot\mu_t|^2\,\d t.
\end{equation}
Now notice that the map $E:C([0,1],X)\to[0,\infty]$ given by
$E(\gamma):=\int_0^1|\dot\gamma_t|^2\,\d t$ if $\gamma\in
AC^2([0,1],X)$ and $+\infty$ otherwise, is lower semicontinuous and,
via a simple equicontinuity argument, with compact sublevels.
Therefore by Prokorov's theorem we get that
$(\ppi^n)\subset\prob{C([0,1],X)}$ is a tight sequence, hence for
any limit measure $\ppi$ the uniform bound
\eqref{eq:persuperposition} gives the thesis.
\end{proof}

\begin{proposition}\label{prop:boundweak}
Let $[0,1]\ni t\mapsto\mu_t=f_t\mm$ be a curve in $AC^2([0,1],(\prob
X,W_2))$.  Assume that for some $0<c<C<\infty$ it holds $c\leq
f_t\leq C$ $\mm$-a.e.~for any $t\in[0,1]$, and that $f_0$ is Sobolev
along a.e.~curve with $\weakgrad{f_0}\in L^2(X,\mm)$. Then
\[
\intX f_0\log f_0\,\d\mm-\intX f_t\log f_t\,\d\mm\leq
\frac12\int_0^t\int_X\frac{\weakgrad{f_0}^2}{f_0^2}f_s\,\d
s\,\d\mm+\frac12\int_0^t|\dot\mu_s|^2\,\d s,\qquad\forall t>0.
\]
\end{proposition}
\begin{proof}
Let $\ppi\in\prob{C([0,1],X)}$ be a plan associated to the curve
$(\mu_t)$ as in Proposition~\ref{prop:lisini}. The assumption
$f_t\leq C$ $\mm$-a.e.~and the fact that
$\iint_0^1|\dot\gamma_t|^2\,\d
t\,\d\ppi(\gamma)=\int|\dot\mu_t|^2\,\d t<\infty$ guarantee that $\ppi$
is a test plan. Now notice that it holds $\weakgrad{\log
f_t}=\weakgrad{f_t}/f_t$ (because $z\mapsto \log z$ is $C^1$ in
$[c,C]$)), thus we get
\[
\begin{split}
\intX f_0\log f_0\,\,\d\mm-\intX f_t\log f_t\,\d\mm&\leq \intX \log f_0(f_0-f_t)\,\d\mm
=\int \Big(\log f_0\circ \e_0-\log f_0\circ \e_t\Big)\,\d\ppi\\
&\leq\iint_0^t\frac{\weakgrad{f_0}(\gamma_s)}{f_0(\gamma_s)}|\dot\gamma_s|\,\d s\,\d\ppi(\gamma)\\
&\leq\frac12\iint_0^t\frac{\weakgrad{f_0}^2(\gamma_s)}{f_0^2(\gamma_s)}\,\d s\,\d\ppi(\gamma)
+\frac12\iint_0^t|\dot\gamma_s|^2\,\d s\,\d\ppi(\gamma)\\
&=\frac12\int_0^t\int_X\frac{\weakgrad{f_0}^2}{f_0^2}f_s\,\d
s\,\d\mm+\frac12\int_0^t|\dot\mu_s|^2\,\d s.
\end{split}
\]
\end{proof}

\subsection{The two notions of gradient coincide}

Here we prove that the two notions of ``norm of weak gradient'' we
introduced coincide. We already noticed in Remark \ref{re:weakrel}
that $\weakgrad f\leq \relgrad f$, so that to conclude we need to
show that $\weakgrad f\geq \relgrad f$.

The key argument to achieve this is the following lemma, which gives
a sharp bound on the $W_2$-speed of the $L^2$-gradient flow of $\C$.
This lemma has been introduced in \cite{GigliKuwadaOhta10} to study
the heat flow on Alexandrov spaces, see also
Section~\ref{se:heatgf}.

\begin{lemma}[Kuwada's lemma]\label{le:kuwada}
Let $f_0\in L^2(X,\mm)$ and let $(f_t)$ be the $L_2$-gradient flow
of $\C$ starting from $f_0$. Assume that for some $0< c\leq
C<\infty$ it holds $c\leq f_0\leq C$ $\mm$-a.e.~in $X$, and that
$\intX f_0\,\d\mm=1$. Then the curve $t\mapsto \mu_t:=f_t\mm$ is
absolutely continuous w.r.t. $W_2$ and it holds
\[
|\dot\mu_t|^2\leq\intX\frac{\relgrad{f_t}^2}{f_t}\,\d \mm,\qquad
\text{for a.e.~$t\in (0,\infty)$.}
\]
\end{lemma}
\begin{proof}
We start from the duality formula \eqref{eq:dualitabase} with
$\varphi=-\psi$: taking into account the factor 2 and using the
identity $Q_1(-\psi)=\psi^c$ we get
\begin{equation}\label{eq:dualityQ}
\frac{W_2^2(\mu,\nu)}2=\sup_\varphi\int_X Q_1\varphi\, d\nu-\int_X\varphi\,d\mu\\
\end{equation}
where the supremum runs among all Lipschitz functions $\varphi$.

Fix such a $\varphi$ and recall
(Proposition~\ref{prop:timederivative}) that the map $t\mapsto
Q_t\varphi$ is Lipschitz with values in $L^\infty(X,\mm)$, and a
fortiori in $L^2(X,\mm)$.

Fix also $0\leq t<s$, set $\ell=(s-t)$ and recall that since $(f_t)$
is the Gradient Flow of $\C$ in $L^2$, the map $[0,\ell]\ni
\tau\mapsto f_{t+\tau}$ is absolutely continuous with values in $L^2$. Therefore
the map $[0,\ell]\ni\tau\mapsto Q_{\frac\tau\ell}\varphi\, f_{t+\tau}$
is absolutely continuous with values in $L^2$. The equality
\[
\frac{Q_{\frac{\tau+h}\ell}\varphi\, f_{t+\tau+h}-Q_{\frac{\tau}\ell}\,\varphi f_{t+\tau}}{h}=
f_{t+\tau}\,\frac{Q_{\frac{\tau+h}\ell}\varphi-Q_{\frac\tau\ell}\varphi }{h}+Q_{\frac{\tau+h}\ell}\varphi
\,\frac{ f_{t+\tau+h}- f_{t+\tau}}{h},
\]
together with the uniform continuity of $(x,\tau)\mapsto Q_{\frac\tau\ell}\varphi(x)$
shows that the derivative of $\tau\mapsto Q_{\GGG\frac\tau\ell}\varphi\, f_{t+\tau}$ can be computed via the Leibniz rule.

We have:
\begin{equation}
\label{eq:step1}
\begin{split}
\int_X Q_1\varphi \,\d\mu_s-\int_X\varphi \,\d\mu_t&=\intX Q_1\varphi f_{t+\ell}\,\d\mm-\int_X\varphi f_t\,\d\mm=
\int_X\int_0^\ell\frac{\d}{\d\tau}\big(Q_{\frac\tau\ell}\varphi f_{t+\tau}\big)\,\d\tau \,\d\mm\\
& \le \int_X\int_0^\ell \Big(-\frac{|D Q_{\frac\tau\ell}\varphi |^2}{2\ell}f_{t+\tau}+Q_{\frac\tau\ell}\varphi
\,\Delta f_{t+\tau}\Big)\,\d\tau \,\d\mm,\\
\end{split}
\end{equation}
having used Theorem~\ref{thm:subsol}. Observe that by inequalities \eqref{eq:delta1} and \eqref{eq:facile} we have
\begin{equation}
\label{eq:sarannouguali}
\begin{split}
\int_X Q_{\frac\tau\ell}\varphi \,\Delta f_{t+\tau} \,\d\mm&\leq
\int_X\relgrad{Q_{\frac\tau\ell}\varphi}\,\relgrad{f_{t+\tau}}\,\d\mm
\leq \int_X|D Q_{\frac\tau\ell}\varphi |\,\,\relgrad{f_{t+\tau}}\,\d\mm\\
&\leq \frac1{2\ell}\int_X|D Q_{\frac\tau\ell}\varphi |^2f_{t+\tau}d\mm+
\frac\ell2\int_X\frac{\relgrad{f_{t+\tau}}^2}{f_{t+\tau}}\,\d\mm.
\end{split}
\end{equation}
Plugging this inequality in \eqref{eq:step1}, we obtain
\[
\int_X Q_1\varphi \,\d\mu_s-\int_X\varphi \,\d\mu_t\leq
\frac\ell2\int_0^\ell\int_X\frac{\relgrad{f_{t+\tau}}^2}{f_{t+\tau}}\,\d\mm.
\]
This latter bound does not depend on $\varphi$, so from \eqref{eq:dualityQ} we deduce
\[
W_2^2(\mu_t,\mu_s)\leq\ell\int_0^\ell\int_X\frac{\relgrad{f_{t+\tau}}^2}{f_{t+\tau}}\,\d\mm.
\]
Since $f_{r}\geq c$ for any $r\geq 0$ and
$r\mapsto\C(f_{r})$ is nonincreasing and finite for every
$r>0$, we immediately get that $t\mapsto\mu_t$ is locally
Lipschitz in $(0,\infty)$. At Lebesgue points of
$t\mapsto\int_X\relgrad{f_t}^2/f_t\,\d\mm$ we obtain the
stated pointwise bound on the metric speed.
\end{proof}

\begin{theorem}\label{thm:graduguali}
Let $f\in L^2(X,\mm)$. Assume that $f$ is Sobolev along a.e.~curve
and that $\weakgrad f\in L^2(X,\mm)$. Then $f\in D(\C)$ and
$\relgrad f=\weakgrad f$ $\mm$-a.e.~in $X$.
\end{theorem}
\begin{proof}
Up to a truncation argument and addition of a constant, we can
assume that $0<c\leq f\leq C<\infty$ $\mm$-a.e.~in $X$ for some
$c,\,C$. Let $(f_t)$ be the $L_2$-gradient flow of $\C$ 
starting from $f$
and recall
that from Proposition~\ref{prop:basecal} we have
\[
\intX f\log f\,\d\mm-\intX f_t\log f_t\,\d\mm=
\int_0^t\int_X\frac{\relgrad{f_s}^2}{f_s}\,\d
s\,\d\mm<\infty\qquad\forevery t>0.
\]
On the other hand, from Proposition~\ref{prop:boundweak} and
Lemma~\ref{le:kuwada} we have
\begin{equation}
\intX f\log f\,\d\mm-\intX f_t\log f_t\,\d\mm\leq
\frac12\int_0^t\int_X\frac{\weakgrad{f}^2}{f^2}f_s\,\d
s\,\d\mm+\frac12\int_0^t\int_X\frac{\relgrad{f_s}^2}{f_s}\,\d
s\,\d\mm.\label{eq:4}
\end{equation}
Hence we deduce
\[
\int_0^t 4\C(\sqrt{f_s})\,\d
s=\frac12\int_0^t\int_X\frac{\relgrad{f_s}^2}{f_s}\,\d s\,\d\mm
\leq\frac12\int_0^t\int_X\frac{\weakgrad f^2}{f^2}f_s\,\d s\,\d\mm.
\]
Letting $t\downarrow 0$, taking into account the $L^2$-lower
semicontinuity of $\C$ and the fact - easy to check from the maximum
principle - that $\sqrt{f_s}\to\sqrt f$ as $s\downarrow 0$ in
$L^2(X,\mm)$, we get $\C(\sqrt f)\leq\limi_{t\downarrow
0}\frac1t\int_0^t\C(\sqrt{f_s})\,\d s$. On the other hand, the bound
$f\geq c>0$ ensures $\frac{\weakgrad{f}^2}{f^2}\in L^1(X,\mm)$ and
the maximum principle again together with the convergence of $f_s$
to $f$ in $L^2(X,\mm)$ when $s\downarrow 0$ grants that the
convergence is also weak$^*$ in $L^\infty(X,\mm)$, therefore
$\intX\frac{\weakgrad f^2}f\,\d\mm=\frac 1t \lim_{t\downarrow
0}\int_0^t\int_X\frac{\weakgrad f^2}{f^2}f_s\,\d\mm\,\d s$.

In summary, we proved
\[
\frac12\intX\frac{\relgrad{f}^2}{f}\,\d\mm\leq\frac12\intX\frac{\weakgrad{f}^2}{f}\,\d\mm,
\]
which, together with the inequality $\weakgrad f\leq\relgrad f$
$\mm$-a.e.~in $X$, gives the conclusion.
\end{proof}

We are now in the position of defining the Sobolev space $W^{1,2}(X,\sfd,\mm)$.
We start with the following simple and general lemma.

\begin{lemma}\label{le:persob}
Let $(B,\|\cdot\|)$ be a Banach space and let $E:B\to[0,\infty]$ be
a 1-homogeneous, convex and lower semicontinuous map. Then the
vector space $\{E<\infty\}$ endowed with the norm
\[
\|v\|_E:=\sqrt{\|v\|^2+E^2(v)},
\]
is a Banach space.
\end{lemma}
\begin{proof}
It is clear that $(D(E),\|\cdot\|_E)$ is a normed space, so we only
need to prove completeness. Pick a sequence $(v_n)\subset D(E)$
which is Cauchy w.r.t. $\|\cdot\|_E$. Then, since
$\|\cdot\|\leq\|\cdot\|_E$ we also get that $(v_n)$ is Cauchy w.r.t.
$\|\cdot\|$, and hence there exists $v\in B$ such that $\|v_n-v\|\to
0$. The lower semicontinuity of $E$ grants that
$E(v)\leq\limi_nE(v_n)<\infty$ and also that  it holds
\[
\lims_{n\to\infty}\|v_n-v\|_E\leq\lims_{n,m\to\infty}\|v_n-v_m\|_E=0,
\]
which is the thesis.
\end{proof}

Therefore, if we want to build the space $W^{1,2}(X,\sfd,\mm)\subset
L^2(X,\mm)$, the only thing that we need is an $L^2$-lower
semicontinuous functional playing the role which on $\R^d$ is played
by the $L^2$-norm of the distributional gradient of Sobolev
functions. We certainly have this functional, namely the map
$f\mapsto\|\relgrad f\|_{L^2(X,\smm)}=\|\weakgrad
f\|_{L^2(X,\smm)}$. Hence the lemma above provides the Banach space
$W^{1,2}(X,\sfd,\mm)$. Notice that in general $W^{1,2}(X,\sfd,\mm)$
is not Hilbert: this is not surprising, as already the Sobolev space
$W^{1,2}$ built over $(\R^d,\|\cdot\|,\mathcal L^d)$ is not Hilbert
if the underlying norm $\|\cdot\|$ does not come from a scalar
product.

\subsection{Comparison with previous approaches}\label{sec:comparegradients}

It is now time to underline that the one proposed here is certainly
not the first definition of Sobolev space over a metric measure
space (we refer to \cite{Heinonen07} for a much broader overview on
the subject). Here we confine the discussion only to weak notions of
(modulus of) gradient, and in particular to \cite{Cheeger00} and
\cite{Koskela-MacManus98,Shanmugalingam00}. 
Also, we discuss only the quadratic case,
referring to \cite{Ambrosio-Gigli-Savare11tris} for
general power functions $p$ and the independence (in a suitable
sense) of $p$ of minimal gradients.

In \cite{Cheeger00} Cheeger proposed a relaxation procedure similar
to the one used in Subsection~\ref{se:rellip}, but rather than
relaxing the local Lipschitz constant of Lipschitz functions, he
relaxed upper gradients of arbitrary functions. More precisely, he
defined
\[
E(f):=\inf\limi_{n\to\infty}\|G_n\|_{L^2(X,\smm)},
\]
where the infimum is taken among all sequences $(f_n)$ converging to
$f$ in $L^2(X,\mm)$ such that $G_n$ is an upper gradient for $f_n$.
Then, with the same computations done in Subsection~\ref{se:rellip}
(actually and obviously, the story goes the other way around: we closely followed his arguments) he
showed that for $f\in D(E)$ there is an underlying notion of weak
gradient $\cgrad f$, called minimal generalized upper gradient, such
that $E(f)=\|\cgrad f\|_{L^2(X,\smm)}$ and
\[
\cgrad f\leq G\qquad\text{$\mm$-a.e.~in $X$,}
\]
for any $G$ weak limit of a sequence $(G_n)$ as in the definition of $E(f)$.

Notice that since the local Lipschitz constant is always an upper gradient for
Lipschitz functions, one certainly has
\begin{equation}
\label{eq:crel} \cgrad f\leq\relgrad f\qquad\text{$\mm$-a.e.~in $X$,
for any $f\in D(\C)$.}
\end{equation}
Koskela and MacManus \cite{Koskela-MacManus98} 
introduced and Shanmugalingam \cite{Shanmugalingam00} further studied a procedure close to
ours (again: actually we have been inspired by them)
to produce a notion of ``norm of weak gradient'' which does not
require a relaxation procedure. Recall that for $\Gamma\subset
AC([0,1],X)$ the 2-Modulus ${\rm Mod}_2(\Gamma)$ is defined by
\begin{equation}
\label{eq:defmod2} {\rm
Mod}_2(\Gamma):=\inf\Big\{\|\rho\|^2_{L^2(X,\mm)}\ : \
\int_\gamma\rho\geq 1\ \ \forall
\gamma\in\Gamma\Big\}\qquad\forevery\Gamma\subset AC([0,1],X).
\end{equation}
It is possible to show that the 2-Modulus is an outer measure on
$AC([0,1],X)$. Building on this notion, Koskela and MacManus
\cite{Koskela-MacManus98} considered the
class of functions $f$ which satisfy the upper gradient inequality
not necessarily along all curves, but only out of a ${\rm
Mod}_2$-negligible set of curves. In order to compare more properly
this concept to Sobolev classes, Shanmugalingam said that
$G:X\to[0,\infty]$ is a weak upper gradient for $f$ if there exists
$\tilde f=f$ $\mm$-a.e.~such that
\[
\big|\tilde f(\gamma_0)-\tilde f(\gamma_1)\big|\leq\int_\gamma
G\qquad\forevery \gamma\in AC([0,1],X)\setminus \mathcal N
\quad\text{with ${\rm Mod}_2(\mathcal N)=0$.}
\]
Then, she defined the energy
$\tilde E:L^2(X,\mm)\to[0,\infty]$ by putting
\[
\tilde E(f):=\inf\|G\|_{L^2(X,\mm)}^2,
\]
where the infimum is taken among all weak upper gradient $G$ of $f$
according to the previous condition. Thanks to the properties of the
2-modulus (a stability property of weak upper gradients analogous to
ours), it is possible to show that $\tilde E$ is indeed $L^2$-lower
semicontinuous, so that it leads to a good definition of the Sobolev
space. Also, using a key lemma due to Fuglede, Shanmugalingam proved
that $E=\tilde E$ on $L^2(X,\mm)$, so that they produce the same definition
of Sobolev space $W^{1,2}(X,\sfd,\mm)$ and the underlying gradient
$|D f|_S$ which gives a pointwise representation to $\tilde
E(f)$ is the same $\cgrad f$ behind the energy $E$.

Observe now that for a Borel set $\Gamma\subset AC^2([0,1],X)$ and a
test plan $\ppi$, integrating w.r.t. $\ppi$ the inequality
$\int_\gamma \rho\geq 1$ $\forall \gamma\in\Gamma$ and then
minimizing over $\rho$, we get
$$
\bigl[\ppi(\Gamma)\bigr]^2\leq C(\ppi){\rm
Mod}_2(\Gamma)\iint_0^1|\dot\gamma|^2\,\d s\,\d\ppi(\gamma),
$$
which shows that any ${\rm Mod}_2$-negligible set of curves is also
negligible according to Definition~\ref{def:testplans}. This fact
easily yields that any $f\in D(\tilde E)$ is Sobolev along a.e.
curve and satisfies
\begin{equation}
\label{eq:weakc} \weakgrad f\leq\cgrad f,\qquad\text{$\mm$-a.e.~in
$X$.}
\end{equation}
Given that we proved in Theorem~\ref{thm:graduguali} that $\relgrad
f=\weakgrad f$, inequalities \eqref{eq:crel} and \eqref{eq:weakc}
also give that $\relgrad f=\weakgrad f=\cgrad f=|D f|_S$ (the
smallest one among the four notions coincides with the largest one).

What we get by the new approach to Sobolev spaces on metric measure spaces
is the following result.

\begin{theorem}[Density in energy of Lipschitz functions]\label{thm:lipdense}
Let $(X,\sfd,\mm)$ be a compact normalized metric measure space.
Then for any $f\in L^2(X,\mm)$ with weak upper gradient in $L^2(X,\mm)$ there
exists a sequence $(f_n)$ of Lipschitz functions converging to $f$
in $L^2(X,\mm)$ such that both $|D f_n|$ and $\weakgrad{f_n}$
converge to $\weakgrad f$ in $L^2(X,\mm)$ as $n\to\infty$.
\end{theorem}
\begin{proof}
Straightforward consequence of the identity of weak and relaxed
gradients and of Proposition~\ref{le:strongappr}.
\end{proof}
Let us point out a few aspects behind the strategy of the
proof of Theorem \ref{thm:lipdense},
which of course strongly relies on 
Lemma \ref{le:kuwada} and Proposition \ref{prop:boundweak}.
First of all, 
let us notice that the stated existence 
of a sequence of Lipschitz function
$f_n$ converging to $f$ with $|D f_n|\to \weakgrad f$ in
$L^2(X,\mm)$ is equivalent to show that
\begin{gather}
  \label{eq:2bis}
  \lim_{n\to\infty} Y_{1/n}(f)\le \intX \weakgrad f^2\,\d\mm,
  \intertext{where, for $\tau>0$, $Y_\tau$ denotes
    the Yosida regularization}
  \notag
  Y_\tau(f):=\,\inf_{h\in \Lip(X)}\left\{\frac
    12\intX|D h|^2\,\d\mm
    +\frac 1{2\tau} \intX |h-f|^2\,\d\mm
  \right\}.
\end{gather}
In fact, the sequence $f_n$ can be chosen by a simple diagonal
argument 
among the approximate minimizers of $Y_{1/n}(f)$.
On the other hand,
it is well known that the relaxation procedure we used to define the
Cheeger energy yields 
\begin{equation}
  \label{eq:3}
  Y_{1/n}(f)=\min_{h\in D(\C)}\left\{\C(h)+\frac n{2}\intX |h-f|^2\,\d\mm\right\},
\end{equation}
and therefore \eqref{eq:2bis} 
could be achieved by trying to estimate the Cheeger energy 
of the unique minimizer $\tilde f_n$ of \eqref{eq:3} in terms of 
$\weakgrad f$. 

Instead of using the Yosida regularization $Y_{1/n}$,
in the proof of Theorem \ref{thm:graduguali} 
we obtained a better approximation of $f$ 
by flowing it 
(for a small time step, say $t_n\downarrow 0$)
through the $L^2$- gradient flow $f_t$ of the Cheeger energy.
This flow is strictly related to $Y_\tau$, since it can
be obtained as the limit of suitably rescaled 
iterated minimizers of $Y_\tau$
(the so called Minimizing Movement scheme, see
e.g.~\cite{Ambrosio-Gigli-Savare08}), but has the 
great advantage to provide a continuous curve of probability densities
$f_t$, which can be represented as 
the image of a test plan, through Lisini's Theorem.
Thanks to this representation and Kuwada's Lemma, 
we were allowed to use 
the weak upper gradient $\weakgrad f$ instead of $\relgrad f$
to estimate the Entropy dissipation along $f_t$ 
(see \eqref{eq:4}) and
to obtain the desired sharp bound of $\relgrad {f_s}$ at least 
for some time $s\in (0,t_n)$. 
In any case, \emph{a posteriori} we recovered the validity of \eqref{eq:2bis}.

This density result was previously known (via the use of maximal
functions and covering arguments) under the assumption that the
space was doubling and supported a local Poincar\'e inequality for
weak upper gradients, see \cite[Theorem~4.14,
Theorem~4.24]{Cheeger00}. Actually, Cheeger proved more, namely that
under these hypotheses Lipschitz functions are dense in the
$W^{1,2}$ norm, a result which is still unknown in the general case.
Also, notice that another byproduct of our density in energy result
is the equivalence of local Poincar\'e inequality stated for
Lipschitz functions on the left hand side and slope on the right
hand side, and local Poincar\'e inequality stated for general
functions on the left hand side and upper gradients on the right
hand side; this result was previously known \cite{Heinonen-Koskela99} under much more restrictive assumptions on
the metric measure structure.

\section{The relative entropy and its $W_2$-gradient flow}\label{se:gfent}

In this section we study the $W_2$-gradient flow of the relative
entropy on spaces with Ricci curvature bounded below (in short:
$CD(K,\infty)$ spaces). The content is essentially extracted from
\cite{Gigli10}. As before the space $(X,\sfd,\mm)$ is compact and
normalized (i.e. $\mm(X)=1$).

Recall that the relative entropy functional $\entv:\prob X\to[0,\infty]$ is defined by
\[
\ent\mu:=\left\{
\begin{array}{ll}
\displaystyle{\int_Xf\log f\,\d\mm}&\textrm{ if }\mu=f\mm,\\
+\infty&\textrm{ otherwise.}
\end{array}
\right.
\]

\begin{definition}[Weak bound from below on the Ricci curvature] \label{def:cdkinfty}
We say that $(X,\sfd,\mm)$ has Ricci curvature bounded from below by
$K$ for some $K\in\R$ if the Relative Entropy functional $\entv$ is
$K$-convex along geodesics in $(\prob X,W_2)$. More precisely, if
for any $\mu_0,\,\mu_1\in D(\entv)$ there exists a constant speed
geodesic $\mu_t:[0,1]\to\Probabilities{X}$ between $\mu_0$ and
$\mu_1$ satisfying
$$
{\rm Ent}_\smm(\mu_t)\leq (1-t){\rm Ent}_\smm(\mu_0)+ t{\rm
Ent}_\smm(\mu_1)-\frac{K}{2}t(1-t)W_2^2(\mu_0,\mu_1)\qquad \forall
t\in [0,1].
$$
\end{definition}

This definition was introduced in \cite{Lott-Villani09} and
\cite{Sturm06I}. Its two basic features are: {\bf compatibility} with
the Riemannian case (i.e. a compact Riemannian manifold endowed with
the normalized volume measure has Ricci curvature bounded below by
$K$ in the classical pointwise sense if and only if $\entv$ is
$K$-geodesically convex in $(\prob X,W_2)$) and {\bf stability}
w.r.t. measured Gromov-Hausdorff convergence.

We also recall that Finsler geometries are included in the class of
metric measure spaces with Ricci curvature bounded below. This means
that if we have a smooth compact Finsler manifold (that is: a
differentiable manifold endowed with a norm - possibly not coming
from an inner product - on each tangent space which varies smoothly
on the base point) endowed with an arbitrary positive $C^\infty$
measure, then this space has Ricci curvature bounded below
by some $K\in\R$ (see the theorem stated at page 926 of
\cite{Villani09} for the flat case and \cite{Ohta-finsler09} for the general
one). 

The goal now is to study the $W_2$-gradient flow of $\entv$. Notice
that the general theory of gradient flows of $K$-convex functionals
ensures the following existence result (see the representation
formula for the slope \eqref{eq:slopegeodetiche} and
Theorem~\ref{thm:edeex}).

\begin{theorem}[Consequences of the general theory of gradient flows]\label{thm:easygf}
Let $(X,\sfd,\mm)$ be a $CD(K,\infty)$ space. Then the slope
$|D^-\entv|$ is lower semicontinuous w.r.t. weak convergence
and for any $\mu\in D(\entv)$ there exists a gradient flow 
(in the EDE sense of Definition \ref{def:ede}) of
$\entv$ starting from $\mu$.
\end{theorem}
Thus, existence is granted. The problem is then to show uniqueness
of the gradient flow. To this aim, we need to introduce the concept
of \emph{push forward via a plan}.

\begin{definition}[Push forward via a plan]\label{def:pushplan}
Let $\mu\in\prob X$ and let $\ggamma\in\prob{X^2}$ be such that
$\mu\ll\pi^1_\sharp\ggamma$. The measures  $\ggamma_\mu\in\prob{X^2}$
and ${\ggamma_\sharp\mu}\in \prob X$
are defined as:
\[
\d\ggamma_\mu(x,y):=\frac{\d\mu}{\d\pi^1_\sharp
  \ggamma}(x)\d\ggamma(x,y),\qquad
{\ggamma_\sharp\mu}:=\pi^2_\sharp \ggamma_\mu.
\]
\end{definition}
Observe that, since $\ggamma_\mu\ll\ggamma$, we have
${\ggamma_\sharp\mu}\ll\pi^2_\sharp \ggamma$. We will say that $\ggamma$
has {bounded \compression{}} if there exist $0<c\leq C<\infty$
such that $c\mm\leq\pi^i_\sharp\ggamma\leq C\mm$, $i=1,2$. Writing
$\mu=\frho\,\pi^1_\sharp\ggamma$, the definition gives that
\begin{equation}
\label{eq:densgammasharp}
\ggamma_\sharp\mu=\eta\,\pi^2_\sharp\ggamma\quad\text{with $\eta$ given by}\quad
\eta(y)=\int\frho(x)\,\d\ggamma_y(x),
\end{equation}
where $\{\ggamma_y\}_{y\in X}$ is the disintegration of $\ggamma$ w.r.t. its second marginal.

The operation of push forward via a plan has interesting properties
in connection with the relative entropy functional.

\begin{proposition}\label{prop:pushgamma}
The following properties hold:
\begin{itemize}
\item[(i)] For any $\mu,\,\nu\in\prob X$, $\ggamma\in\prob{X^2}$ such that $\mu,\,\nu\ll\pi^1_\sharp\ggamma$
it holds
\[\entr{\ggamma_\sharp\mu}{\sggamma_\sharp\nu}\leq\entr\mu\nu.
\]
\item[(ii)] For $\mu\in D(\entv)$ and $\ggamma\in\prob{X^2}$ with bounded \compression{}, it holds $\ggamma_\sharp\mu\in
D(\entv)$.
\item[(iii)] Given $\ggamma\in\prob{X^2}$ with bounded \compression{}, the map
\[
D(\entv)\ni\mu\qquad\mapsto \qquad\entr\mu{\smm}-\entr{\ggamma_\sharp\mu}{\smm},
\]
is convex (w.r.t. linear interpolation of measures).
\end{itemize}
\end{proposition}
\begin{proof}\\* {\bf (i)}. We can assume $\mu\ll\nu$, otherwise
there is nothing to prove. Then it is immediate to check from the
definition that $\ggamma_\sharp\mu\ll\ggamma_\sharp\nu$. Let
$\mu=\frho\nu$, 
$\nu=\theta\, \pi^1_\sharp\gamma,$
$\ggamma_\sharp\mu=\eta\,\ggamma_\sharp\nu$, and
$u(z):=z\log z$. By disintegrating $\ggamma$ as in
\eqref{eq:densgammasharp}, we have that
\begin{displaymath}
  \eta(y)=\int \frho(x)\,\d\tilde\ggamma_y(x),\quad
  \tilde\ggamma_y=\Big(\int \theta(x)\,\d\ggamma_y(x)\Big)^{-1}\,\theta\,\ggamma_y.
\end{displaymath}
The convexity of $u$ and Jensen's inequality with the probability
measures $\tilde\ggamma_y$ yield
\begin{displaymath}
  u(\eta(y))\le \int u(\frho(x))\,\d\tilde\ggamma_y(x).
\end{displaymath}
Since $\{\tilde\ggamma_y\}_{y\in X}$ is the disintegration of
$\tilde\ggamma=(\theta\circ\pi^1)\ggamma$ with respect to its second
marginal $\ggamma_\sharp\nu$ and the first marginal of $\tilde\ggamma$ is $\nu$, by
integration of both sides with respect to $\ggamma_\sharp\nu$  we get
\[
\begin{split}
\entr{\ggamma_\sharp\mu}{\sggamma_\sharp\nu}&=
\int u(\eta(y))\,\d\ggamma_\sharp\nu(y)\le \int \Big(\int u(\frho(x))\,\d\tilde\ggamma_y(x) \Big)\,\d\ggamma_\sharp\nu(y)\\
&\leq \int u(\frho(x))\,\d\tilde\ggamma(x,y)=
\int u(\frho(x))\,\d\nu(x)=\entr{\mu}{\nu}.
\end{split}
\]
\noindent{\bf (ii)}. Taking into account the identity
\begin{equation}
\label{eq:scambioentropie}
\entr\mu\nu=\entr\mu\sigma+\int\log\Big(\frac{\d\sigma}{\d\nu}\Big)\,\d\mu,
\end{equation}
valid for any $\mu,\,\nu,\sigma\in\prob X$ with $\sigma$ having
bounded density w.r.t. $\nu$, the fact that
$\ggamma_\sharp(\pi^1_\sharp\ggamma)=\pi^2_\sharp\ggamma$ and the
fact that $c\mm\leq\pi^1_\sharp \ggamma,\pi^2_\sharp \ggamma\leq C\mm$, the
conclusion follows from
\[
\entr{\ggamma_\sharp\mu}\smm\leq
\entr{\ggamma_\sharp\mu}{\pi^2_\sharp\sggamma} +\log
C\leq\entr\mu{\pi^1_\sharp\sggamma}+\log C\leq\entr\mu\smm+\log
C-\log c.
\]
\noindent{\bf (iii)}. Let $\mu_0,\,\mu_1\in D(\entv)$ and define
$\mu_t:=(1-t)\mu_0+t\mu_1$ and $\nu_t:=\ggamma_\sharp\mu_t$. A
direct computation shows that
\[
\begin{split}
(1-t){\rm Ent}_{\smm}(\mu_0)+t{\rm Ent}_{\smm}(\mu_1)-{\rm Ent}_{\smm}(\mu_t)
&=(1-t){\rm Ent}_{\mu_t}(\mu_0)+t{\rm Ent}_{\mu_t}(\mu_1),\\
(1-t){\rm Ent}_{\smm}(\nu_0)+t{\rm Ent}_{\smm}(\nu_1)-{\rm Ent}_{\smm}(\nu_t)
&=(1-t){\rm Ent}_{\nu_t}(\nu_0)+t{\rm Ent}_{\nu_t}(\nu_1),
\end{split}
\]
and from $(i)$ we have that
\[
\entr{\mu_i}{\mu_t}\geq\entr{\ggamma_\sharp\mu_i}{\sggamma_\sharp\mu_t}
=\entr{\nu_i}{\nu_t},\qquad\forall t\in [0,1],\ i=0,1,
\]
which gives the conclusion.
\end{proof}

In the next lemma and in the sequel we use the short notation
$$
C(\ggamma):=\int_{X\times X}\sfd^2(x,y)\,\d\ggamma(x,y).
$$

\begin{lemma}[Approximability in Entropy and distance]
\label{le:appr} Let $\mu,\,\nu\in D(\entv)$. Then there exists a
sequence $(\ggamma^n)$ of plans with bounded \compression{} such that
$\ent{\ggamma^n_\sharp\mu}\to\ent\nu$ and $C(\ggamma^n_\mu)\to
W_2^2(\mu,\nu)$ as $n\to\infty$.
\end{lemma}
\begin{proof} Let $f$ and $g$ respectively be the densities of $\mu$ and
  $\nu$ w.r.t.\ $\mm$;  
  pick $\ggamma\in\opt(\mu, \nu)$ and, for every $n\in\N$, let $A_n:=\{(x,y):f(x)+g(y)\leq n\}$ and
\[
\ggamma_n:=c_n\left(\ggamma\restr{A_n}+\frac1n(\Id,\Id)_\sharp \mm\right),
\]
where $c_n\to 1$ is the normalization constant. It is immediate to
check that $\ggamma_n$ is of bounded \compression{} and that this
sequence satisfies the thesis (see \cite{Gigli10} for further
details).
\end{proof}

\begin{proposition}[Convexity of the squared slope]\label{prop:slopeconvessa}
Let  $(X,\sfd,\mm)$ be a $CD(K,\infty)$ space. Then the map
\[
D(\entv)\ni\mu\qquad\mapsto\qquad|D^-\entv|^2(\mu)
\]
is convex (w.r.t. linear interpolation of measures).
\end{proposition}

Notice that the only assumption that we make is the $K$-convexity of
the entropy w.r.t. $W_2$, and from this we deduce the convexity
w.r.t. the classical linear interpolation of measures of the squared
slope.

\begin{proof} Recall that from \eqref{eq:slopegeodetiche} we know that
\[
|D^-\entv|(\mu)=\sup_{\nu\in\probt X\atop \nu\neq\mu}
\frac{\left[\ent\mu-\ent\nu-\frac{K^-}2W^2_2(\mu,\nu)\right]^+}{W_2(\mu,\nu)}.
\]
We claim that it also holds
\[
|D^-\entv|(\mu)=\sup_{\sggamma}\frac{\left[\ent\mu-\ent{\ggamma_\sharp\mu}-
\frac{K^-}2C(\ggamma_\mu)\right]^+}{\sqrt{C(\ggamma_\mu)}},
\]
where the supremum is taken among all plans with bounded \compression{}
(where the right hand side is taken 0 by definition if 
$C(\ggamma_\mu)>0$). 

Indeed, Lemma~\ref{le:appr} gives that the first expression is not
larger than the second. For the converse inequality we can assume
$C(\ggamma_\mu)>0$, $\nu=\ggamma_\sharp \mu\neq\mu$, and $K<0$.
Then it is
sufficient to apply the simple inequality
\[
a,\,b,\,c\in\R,\ \  0<b\leq
c\qquad\Rightarrow\qquad\frac{(a-b)^+}{\sqrt
b}\geq\frac{(a-c)^+}{\sqrt c},
\]
with $a:=\ent\mu-\ent{\ggamma_\sharp\mu}$, $b:=\frac{K^-}2W^2_2(\mu,\ggamma_\sharp\mu)$
and $c:=\frac{K^-}2C(\ggamma_\mu)$.

Thus, to prove the thesis it is  enough to show that for every
$\ggamma$ with bounded \compression{} the map
\[
D(\entv)\ni\mu\qquad\mapsto\qquad\frac{\bigl[\bigl(\ent\mu-\ent{\ggamma_\sharp\mu}
-\frac{K^-}2C(\ggamma_\mu)\bigr)^+\bigr]^2}{C(\ggamma_\mu)},
\]
is convex w.r.t. linear interpolation of measures.

Clearly the map
\[
D(\entv)\ni\mu\qquad\mapsto\qquad C(\ggamma_\mu)=\int\left( \int \sfd^2(x,y)\,\d\ggamma_{x}(y)\right)\,\d\mu(x),
\]
where $\{\ggamma_x\}$ is the disintegration of $\ggamma$ w.r.t. its
first marginal, is linear. Thus, from $(iii)$ of
Proposition~\ref{prop:pushgamma} we know that the map
\[
\mu\qquad\mapsto\qquad \ent\mu-\ent{\ggamma_\sharp\mu}-\frac{K^-}2C(\ggamma_\mu),
\]
is convex w.r.t. linear interpolation of measures. Hence the same is
true for its positive part. The conclusion follows from the fact
that the function $\Psi:[0,\infty)^2\to\R\cup\{+\infty\}$ defined by
\[
\Psi(a,b):=\left\{
\begin{array}{ll}
\dfrac{a^2}b,&\qquad\textrm{ if }b>0,\\
+\infty&\qquad\textrm{ if }b=0,a>0\\
0&\qquad\textrm{ if }a=b=0,\\
\end{array}
\right.
\]
is convex and it is nondecreasing w.r.t. $a$.
\end{proof}

The convexity of the squared slope allows to prove uniqueness of the gradient flow of the entropy:

\begin{theorem}[Uniqueness of the gradient flow of $\entv$]\label{thm:gfent}
Let $(X,\sfd,\mm)$ be a $CD(K,\infty)$ space and let $\mu\in
D(\entv)$. Then there exists a unique gradient flow of $\entv$
starting from $\mu$ in $(\prob X,W_2)$.
\end{theorem}
\begin{proof} We recall (inequality \eqref{eq:w2conve}) that the squared Wasserstein distance is convex
w.r.t. linear interpolation of measures. Therefore, given two
absolutely continuous curves $(\mu^1_t)$ and $(\mu^2_t)$, the curve
$t\mapsto \mu_t:=\frac{\mu^1_t+\mu^2_t}{2}$ is absolutely continuous
as well and its metric speed can be bounded from above by
\begin{equation}
\label{eq:w2conv}
|\dot{\mu}_t|^2\leq\frac{|\dot\mu^1_t|^2+|\dot\mu^2_t|^2}{2},\qquad
\text{for a.e.~$t\in (0,\infty)$.}
\end{equation}
Let $(\mu^1_t)$ and $(\mu^2_t)$ be gradient flows of $\entv$
starting from $\mu\in D(\entv)$. Then we have
\[
\begin{split}
\ent{\mu}&=\ent{\mu^1_T}+\frac12\int_0^T|\dot\mu^1_t|^2\,\d t+
\frac12\int_0^T|D^-\entv|^2(\mu^1_t)\,\d t,\quad\forall T\geq 0,\\
\ent{\mu}&=\ent{\mu^2_T}+\frac12\int_0^T|\dot{\mu}^2_t|^2\,\d t+
\frac12\int_0^T|D^-\entv|^2(\mu^2_t)\,\d t,\quad\forall T\geq 0.
\end{split}
\]
Adding up these two equalities, using the convexity of the squared
slope guaranteed by Proposition~\ref{prop:slopeconvessa}, the
convexity of the squared metric speed given by \eqref{eq:w2conv} and
the \emph{strict} convexity of the relative entropy, we deduce that
for the curve $t\mapsto\mu_t$ 
it holds
\[
\ent{\mu}>\ent{\mu_T}+\frac12\int_0^T|\dot{\mu}_t|^2\,\d t+\frac12\int_0^T|D^-\entv|^2(\mu_t)\,\d t,
\]
for every $T$ such that $\mu^1_T\neq \mu^2_T$. This contradicts inequality \eqref{eq:boundtuttecurve}.
\end{proof}

\section{The heat flow as gradient flow}\label{se:heatgf}

It is well known that on $\R^d$ the heat flow can be seen both as
gradient flow of the Dirichlet energy in $L^2$ and as gradient flow
of the relative entropy in $(\probt{\R^d},W_2)$. It is therefore
natural to ask whether this identification between the two a priori
different gradient flows persists or not
in a general compact and normalized metric measure space $(X,\sfd,\mm)$. 

The strategy consists in considering a gradient flow $(f_t)$ of $\C$
with nonnegative initial data and in proving that the curve
$t\mapsto \mu_t:=f_t\mm$ is a gradient flow of $\ent\cdot$ in
$(\prob X,W_2)$: by the uniqueness result of Theorem \ref{thm:gfent}
this will be sufficient to conclude.

We already built most of the ingredients needed for the proof to
work, the only thing that we should add is the following lemma,
where the slope of $\entv$ is bounded from above in terms of the
notions of ``norm of weak gradient'' that we discussed in
Chapter~\ref{se:Sobolev}. Notice that the bound \eqref{eq:slopelip} for Lipschitz functions was already known to Lott-Villani (\cite{Lott-Villani09}), so that our added value here is the use of the density in energy of Lipschitz functions to get the correct, sharp inequality \eqref{eq:5} (sharpness will be seen in \eqref{eq:slopesharp}).
\begin{lemma}[Fisher bounds slope]\label{prop:fishslope}
  Let $(X,\sfd,\mm)$ be a compact and normalized $CD(K,\infty)$ metric-measure
  space  and
  let $\frho$ be a probability density
  which is Sobolev along a.e.\ curve.
 Then
 \begin{equation}
|D^-\entv|^2(\frho\mm)\leq
\intX\frac{\weakgrad\frho^2}{\frho}\,\d\mm =4\intX \weakgrad{\sqrt \frho}^2\,\d\mm.\label{eq:5}
\end{equation}
\end{lemma}
\begin{proof}
Assume at first that $\frho$ is Lipschitz with $0<c\le \frho$, and let $(\frho_n)$ be a
sequence of probability densities such that
$W_2(\frho_n\mm,\frho\mm)\to0$ and where the slope of $\entv$ at
$\frho\mm$ is attained. Choose $\ggamma_n\in\opt(\frho\mm,\frho_n\mm)$
and notice that
\begin{equation}
\label{eq:fishlip}
\begin{split}
\intX\frho\log\frho\,\d\mm&-\intX\frho_n\log\frho_n\,\d\mm\leq\intX(\frho-\frho_n)\log\frho\,\d\mm\\
&=\int\Big(\log\frho(x)-\log\frho(y)\big)\,\d\ggamma_n(x,y)\\
&\leq\sqrt{\int\frac{\big(\log\frho(x)-\log\frho(y)\big)^2}{\sfd^2(x,y)}\,\d\ggamma_n(x,y)}\sqrt{\int
\sfd^2(x,y)\,\d\ggamma_n(x,y)}\\
&= \Big(\int \Big(\int L^2(x,y)\,\d\ggamma_{n,x}(y)\Big)\frho(x)\,\d\mm(x)\Big)^{1/2}\,W_2(\frho\mm,\frho_n\mm),
\end{split}
\end{equation}
where $\ggamma_{n,x}$ is the disintegration of $\ggamma_n$ with
respect to $\frho\mm$, and $L$ is the bounded Borel function
\[
L(x,y):=\left\{
\begin{array}{ll}
\dfrac{\big|\log\frho(x)-\log\frho(y)\big|}{\sfd(x,y)},&\qquad\textrm{if }x\neq y,\\
|D\log\frho|(x)=\dfrac{|D\frho|(x)}{\frho(x)}&\qquad\textrm{if
}x= y.
\end{array}
\right.
\]
Notice that for every $x\in X$ the map $y\mapsto L(x,y)$ is
upper-semicontinuous; since $\int\big(\int
\sfd^2(x,y)\,\d\ggamma_{n,x}\big)\frho(x)\,\d\mm\to 0$ as $n\to\infty$,
we can assume without loss of generality that 
\begin{displaymath}
  \lim_{n\to\infty}\int \sfd^2(x,y)\,\d\ggamma_{n,x}(y)=0\qquad
  \text{for $\frho\mm$-a.e.\ $x\in X$}.
\end{displaymath}
Fatou's Lemma then yields
\[
\lims_{n\to\infty}\int L^2(x,y)\,\d\ggamma_n(x,y)
\leq\intX L^2(x,x)\frho(x)\,\d\mm(x)=\intX\frac{|D\frho|^2}{\frho}\,\d\mm,
\]
hence \eqref{eq:fishlip} gives
\begin{equation}
\label{eq:slopelip}
|D^-\entv|(\frho\mm)=\lims_{n\to\infty}\frac{(\ent{\frho\mm}-\ent{\frho_n\mm})^+}
{W_2(\frho\mm,\frho_n\mm)}\leq\sqrt{ \intX \frac{|D\frho|^2}\frho \,\d\mm}.
\end{equation}
We now turn to the general case. Let  $\frho$ be any
probability density Sobolev along a.e.\ curve such that $\sqrt
\frho\in D(\C)$ (otherwise is nothing
to prove).
We use Theorem~\ref{thm:lipdense} to find a
sequence of Lipschitz functions 
$(\sqrt {\frho_n})$ converging to $\sqrt \frho$ in
$L^2(X,\mm)$ and such that $|D \sqrt {\frho_n}|\to\weakgrad{\sqrt\frho}$ in
$L^2(X,\mm)$ and $\mm$-a.e.. Up to summing up positive and vanishing
constants and multiplying for suitable normalization factors, we can
assume that $0<c_n\leq\frho_n$
and $\intX\frho_n\,\d\mm=1$, for any $n\in\N$. The
conclusion follows passing to the limit in \eqref{eq:slopelip} by
taking into account the weak lower semicontinuity of
$|D^-\entv|$ (formula \eqref{eq:slopegeodetiche} and discussion
thereafter).
\end{proof}

\begin{theorem}[The heat flow as gradient flow]\label{thm:main}
Let $f_0\in L^2(X,\mm)$ be such that
$\mu_0=f_0\mm\in\Probabilities{X}$ and denote by $(f_t)$ the gradient flow
of $\C$ in $L^2(X,\mm)$ starting from $f_0$ and by $(\mu_t)$ the gradient flow
of $\entv$ in $(\prob X,W_2)$ starting from
$\mu_0$.  Then $\mu_t=f_t\mm$ for any $t\geq 0$.
\end{theorem}
\begin{proof}
Thanks to the uniqueness result of Theorem~\ref{thm:gfent}, it is
sufficient to prove that $(f_t\mm)$ satisfies the Energy Dissipation
Equality for $\entv$ in $(\prob X,W_2)$. We assume first that
$0<c\leq f_0\leq C<\infty$ $\mm$-a.e.~in $X$, so that the maximum
principle (Proposition~\ref{prop:basecal}) ensures $0<c\leq f_t\leq
C<\infty$ for any $t>0$. By Proposition~\ref{prop:basecal} we know
that $t\mapsto\entr{f_t\mm}\mm$ is absolutely continuous with
derivative equal to $-\intX\frac{\weakgrad {f_t}^2}{f_t}\,\d\mm$.
Lemma~\ref{le:kuwada} ensures that $t\mapsto f_t\mm$ is absolutely
continuous w.r.t. $W_2$ with squared metric speed bounded by
$\intX\frac{\weakgrad {f_t}^2}{f_t}\,\d\mm$, so that taking into
account Lemma~\ref{prop:fishslope} we get
\[
\entr{f_0\mm}\mm\geq\entr{f_t\mm}\mm+
\frac12\int_0^t|\dot{f_s\mm}|^2\,\d s+\frac12\int_0^t|D^-\entv|^2(f_s\mm)\,\d s,
\]
which, together with \eqref{eq:boundtuttecurve}, ensures the thesis.

For the general case we argue by approximation, considering
$f^n_0:=c_n\min\{n,\max\{f_0,1/n\}\}$, $c_n$ being the normalizing
constant, and the corresponding gradient flow $(f^n_t)$ of $\C$. The
fact that $f^n_0\to f_0$ in $L^2(X,\mm)$ and the convexity of $\C$
implies that $f^n_t\to f_t$ in $L^2(X,\mm)$ for any $t>0$. In
particular, $W_2(f^n_t\mm,f_t\mm)\to 0$ as $n\to\infty$ for
every $t$ (because convergence w.r.t. $W_2$ is equivalent to weak
convergence of measures).

Now notice that we know that
\[
\ent{f^n_0\mm}=\ent{f^n_t}+\frac12\int_0^t|\dot{f^n_s\mm}|^2\,\d s
+\frac12\int_0^t|D^-\entv|^2(f^n_s)\,\d s,\qquad\forall t>0.
\]
Furthermore, it is immediate to check that
$\ent{f^n_0\mm}\to\ent{f_0\mm}$ as $n\to\infty$. The pointwise
convergence of $f^n_t\mm$ to $f_t\mm$ w.r.t. $W_2$ easily yields
that the terms on the right hand side of the last equation are lower
semicontinuous when $n\to\infty$ (recall Theorem~\ref{thm:easygf}
for the slope). Thus it holds
\[
\ent{f_0\mm}\geq\ent{f_t}+\frac12\int_0^t|\dot{f_s\mm}|^2\,\d s+
\frac12\int_0^t|D^-\entv|^2(f_s)\,\d s,\qquad\forall t>0,
\]
which, by \eqref{eq:defedegeq}, is the thesis.

\noindent We know, by Theorem~\ref{thm:gfent}, that there is at most
a gradient flow starting from $\mu_0$. We also know that a gradient
flow $f_t'$ of $\C$ starting from $f_0$ exists, and part $(i)$ gives
that $\mu_t':=f_t'\mm$ is a gradient flow of ${\rm Ent}_\smm$. The
uniqueness of gradient flows gives $\mu_t=\mu_t'$ for all $t\geq 0$.
\end{proof}
As a consequence of the previous Theorem \ref{thm:main} it would not
be difficult to prove that the inequality \eqref{eq:5} is in fact an
identity: if $(X,\sfd,\mm)$ is a compact and normalized $CD(K,\infty)$
space, then $|D^-\entv|(\frho\mm)<\infty$  if and only if the
probability density $\frho$ 
is Sobolev along a.e.\ curve and $\sqrt \frho\in D(\C)$; in this case
\begin{equation}
\label{eq:slopesharp}
  |D^-\entv|^2(\frho\mm)=
\intX\frac{\weakgrad\frho^2}{\frho}\,\d\mm =4\intX \weakgrad{\sqrt \frho}^2\,\d\mm.
\end{equation}

\section{A metric Brenier theorem}

In this section we state and prove the metric Brenier theorem in
$CD(K,\infty)$ spaces we announced in the introduction. It was
recently proven in \cite{Gigli12} that under an
additional non-branching assumption one can really recover an
optimal transport map, see also \cite{Ambrosio-Rajala12} for related results, obtained under stronger
non-branching assumptions and weaker convexity assumptions.

\begin{definition}[Strong $CD(K,\infty)$ spaces]\label{def:strongcd}
We say that a compact normalized metric measure space $(X,\sfd,\mm)$
is a strong $CD(K,\infty)$ space if for any $\mu_0,\,\mu_1\in
D(\entv) $  there exists $\ppi\in\gopt(\mu_0,\mu_1)$ with the following
property. For any bounded Borel function $F:\geo(X)\to[0,\infty)$
such that $\int F\,\d\ppi=1$, it holds
\[
\entr{\mu^F_t}\mm\leq(1-t)\entr{\mu^F_0}\mm+t\entr{\mu^F_1}\mm-\frac{K}{2}t(1-t)W_2^2(\mu_0^F,\mu_1^F),
\]
where $\mu^F_t:=(\e_t)_\sharp(F\ppi)$, for any $t\in[0,1]$.
\end{definition}

Thus, the difference between strong $CD(K,\infty)$ spaces and
standard $CD(K,\infty)$ ones is the fact that geodesic convexity is
required along \emph{all} geodesics induced by the weighted plans
$F\ppi$, rather than the one induced by $\ppi$ only. Notice that the
necessary and sufficient optimality conditions ensure that $(\e_0,\e_1)_\sharp\ppi$ is
concentrated on a $c$-monotone set, hence $
(\e_0,\e_1)_\sharp (F\ppi)$ has the same
property and it is optimal, relative to its marginals. (We remark that recent results of Rajala \cite{Rajala11} suggest that it is not necessary to assume this stronger convexity to get the metric Brenier theorem - and hence not even a treatable notion of spaces with Riemannian Ricci curvature bounded from below - see \cite{AGSMR12} for progresses in this direction)

It is not clear to us whether the notion of being strong
$CD(K,\infty)$ is stable or not w.r.t. measured Gromov-Hausdorff
convergence and, as such, it should be handled with care. The
importance of strong $CD(K,\infty)$ bounds relies on the fact that
on these spaces geodesic interpolation between bounded probability
densities is made of bounded densities as well, thus granting the
existence of many test plans.

Notice that non-branching $CD(K,\infty)$ spaces are always strong
$CD(K,\infty)$ spaces, indeed let $\mu_0,\,\mu_1\in D(\entv)$ and
pick $\ppi\in\gopt(\mu_0,\mu_1)$ such that $\entv$ is $K$-convex
along $((\e_t)_\sharp\ppi)$. From the non-branching hypothesis it
follows that for $F$ as in Definition~\ref{def:strongcd} there
exists a unique element in $\gopt(\mu^F_t,\mu^F_1)$ (resp. in
$\gopt(\mu^F_t,\mu^F_0)$). Also, since $F$ is bounded, from
$\mu_t\in D(\entv)$ we deduce $\mu^F_t\in D(\entv)$. Hence the map
$t\mapsto\entr{\mu^F_t}{\mm}$ is $K$-convex and bounded on
$[\eps,1]$ and on $[0,1-\eps]$ for all $\eps\in (0,1)$, and
therefore it is $K$-convex on $[0,1]$.

\begin{proposition}[Bound on geodesic interpolant]\label{prop:boundlinfty}
Let $(X,\sfd,\mm)$ be a strong $CD(K,\infty)$ space and let
$\mu_0,\,\mu_1\in\prob X$ be with bounded densities. Then there
exists a {\em test} plan $\ppi\in\gopt(\mu_0,\mu_1)$ so that 
the induced geodesic $\mu_t=(\e_t)_\sharp\ppi$ connecting $\mu_0$ to $\mu_1$ is 
made of measures with
uniformly bounded densities.
\end{proposition}
\begin{proof}
Let $M$ be an upper bound on the densities of $\mu_0,\,\mu_1$,
$\ppi\in\gopt(\mu_0,\mu_1)$ be a plan which satisfies the
assumptions of Definition~\ref{def:strongcd} and
$\mu_t:=(\e_t)_\sharp\ppi$. We claim that the measures $\mu_t$ have
uniformly bounded densities. The fact that $\mu_t\ll\mm$ is obvious
by geodesic convexity, so let $\frho_t$ be the density of $\mu_t$ and
assume by contradiction that for some $t_0\in[0,1]$ it holds
\begin{equation}
\label{eq:perassurdo} \frho_{t_0}(x)> Me^{K^-{\rm
D}^2/8},\qquad\forall x\in A,
\end{equation}
where $\mm(A)>0$ and ${\rm D}$ is the diameter of $X$.
Define $\tilde\ppi:=c\ppi\restr{\e_{t_0}^{-1}(A)}$, where $c$ is the normalizing constant
(notice that $\tilde\ppi$ is well defined, because $\ppi(\e_{t_0}^{-1}(A))=\mu_{t_0}(A)>0$)
and observe that the density of $\tilde\ppi$ w.r.t. $\ppi$ is bounded.
Let $\tilde\mu_t:=(\e_t)_\sharp\tilde\ppi$ and $\tilde\frho_t$ its density w.r.t. $\mm$.
From \eqref{eq:perassurdo} we get $\tilde\frho_{t_0}=c\frho_{t_0}$ on $A$ and $\tilde\frho_{t_0}=0$ on $X\setminus A$,
hence
\begin{equation}
\label{eq:boundentrhot}
\entr{\tilde\mu_{t_0}}\mm=\int\log(\tilde\frho_{t_0}\circ\e_{t_0})\,\d\ppi>\log c+\log M+\frac{K^-}{8}{\rm D}^2.
\end{equation}
On the other hand, we have $\tilde\frho_0\leq c\frho_0\leq cM$ and
$\tilde\frho_1\leq c\frho_1\leq cM$ and thus
\begin{equation}
\label{eq:boundentrho01}
\entr{\tilde\mu_i}\mm=\int\log(\tilde\frho_{i}\circ\e_i)\,\d\tilde\ppi\leq
\log c+\log M,\qquad i=0,1.
\end{equation}
Finally, it certainly holds $W_2^2(\tilde\mu_0,\tilde\mu_1)\leq {\rm
D}^2$, so that \eqref{eq:boundentrhot} and \eqref{eq:boundentrho01}
contradict the $K$-convexity of $\entv$ along $(\tilde\mu_t)$. Hence
\eqref{eq:perassurdo} is false and the $\frho_t$'s are uniformly
bounded.
\end{proof}
An important consequence of this uniform bound is the following metric
version of Brenier's theorem. 
\begin{theorem}[A metric Brenier theorem]\label{prop:potweak} Let
  $(X,\sfd,\mm)$ be a strong $CD(K,\infty)$ space,
  let $\frho_0,\,\frho_1$ be 
probability densities and $\varphi$ any
Kantorovich potential for the couple $(\frho_0\mm,\frho_1\mm)$. Then
for every $\ppi\in\gopt(\frho_0\mm,\frho_1\mm)$ it holds
\begin{equation}
\sfd(\gamma_0,\gamma_1)=\weakgrad\varphi(\gamma_0)=|D^+\varphi|(\gamma_0),\qquad\text{for
$\ppi$-a.e.~$\gamma$.}\label{eq:8}
\end{equation}
In particular,
\[
W_2^2(\frho_0\mm,\frho_1\mm)=\int_X\relgrad\varphi^2\,\frho_0\,\d\mm.
\]
If moreover $\frho_0,\frho_1\in L^\infty(X,\mm)$ and $\ppi$ is a
test plan (such a plan exists thanks to Proposition
\ref{prop:boundlinfty}) then
\begin{equation}
  \label{eq:6}
  \lim_{t\downarrow0}
  \frac{\varphi(\gamma_0)-\varphi(\gamma_t)}{\sfd(\gamma_0,\gamma_t)}=|D^+\varphi|(\gamma_0)\quad\text{in }L^2(\mathrm{Geo}(X),\ppi).
\end{equation}
\end{theorem}
\begin{proof}
$\varphi$ is Lipschitz, therefore $|D^+\varphi|$ is an upper
gradient of $\varphi$, and hence
$\weakgrad\varphi\leq|D^+\varphi|$ $\mm$-a.e.. Now fix $x\in X$
and pick any $y\in\partial^c\varphi(x)$. From the $c$-concavity of
$\varphi$ we get
\[
\begin{split}
\varphi(x)&=\frac{\sfd^2(x,y)}2-\varphi^c(y),\\
\varphi(z)&\leq\frac{\sfd^2(z,y)}2-\varphi^c(y)\qquad\forall z\in X.
\end{split}
\]
Therefore
\[
\varphi(z)-\varphi(x)\leq\frac{\sfd^2(z,y)}2-\frac{\sfd^2(x,y)}2\leq \sfd(z,x)\frac{\sfd(z,y)+\sfd(x,y)}2.
\]
Dividing by $\sfd(x,z)$ and letting $z\to x$, by the arbitrariness
of $y\in\partial^c\varphi(x)$ and the fact that
$\supp((\e_0,\e_1)_\sharp\ppi)\subset\partial^c\varphi$ we get
\[
|D^+\varphi|(\gamma_0)\leq
\min_{y\in\partial^c\varphi(\gamma_0)}\sfd(\gamma_0,y)
\leq\sfd(\gamma_0,\gamma_1)\qquad\text{for $\ppi$-a.e.~$\gamma$.}
\]
Since
$$
\intX\weakgrad\varphi^2\frho_0\,\d\mm\leq
\int|D^+\varphi|^2(\gamma_0)\,\d\ppi \quad\text{and}\quad
\int\sfd^2(\gamma_0,\gamma_1)\,\d\ppi(\gamma)=W_2^2(\frho_0\mm,\frho_1\mm),
$$
to conclude it is sufficient to prove that
\begin{equation}
\label{eq:perbre}
W_2^2(\frho_0\mm,\frho_1\mm)\leq\intX\weakgrad\varphi^2\frho_0\,\d\mm.
\end{equation}
Now assume that $\frho_0$ and $\frho_1$ are bounded from above and let
$\tilde\ppi\in\gopt(\frho_0\mm,\frho_1\mm)$ be a test plan (such
$\tilde\ppi$ exists thanks to Proposition~\ref{prop:boundlinfty}).
Since $\varphi$ is a Kantorovich potential and
$(\e_0,\e_1)_\sharp\tilde\ppi$ is optimal, it holds $\gamma_1\in
\partial^c\varphi(\gamma_0)$ for any $\gamma\in\supp(\tilde\ppi)$.
Hence arguing as before we get
\begin{equation}
\varphi(\gamma_0)-\varphi(\gamma_t)\geq\frac{\sfd^2(\gamma_0,\gamma_1)}2-\frac{\sfd^2(\gamma_t,\gamma_1)}2=
\sfd^2(\gamma_0,\gamma_1)\big(t-t^2/2\big).\label{eq:7}
\end{equation}
Dividing by $\sfd(\gamma_0,\gamma_t)=t\sfd(\gamma_0,\gamma_1)$,
squaring and integrating w.r.t. $\tilde\ppi$ we obtain
\begin{equation}
\label{eq:prelimite}
\limi_{t\downarrow0}\int\left(\frac{\varphi(\gamma_0)-\varphi(\gamma_t)}
{\sfd(\gamma_0,\gamma_t)}\right)^2\,\d\tilde\ppi(\gamma)\geq
\int\sfd^2(\gamma_0,\gamma_1)\,\d\tilde\ppi(\gamma)
=W_2^2(\frho_0\mm,\frho_1\mm).
\end{equation}
Using Remark~\ref{re:restr} and the fact that $\tilde\ppi$ is a test
plan we have
\begin{equation}
\label{eq:intermediobrenier}
\begin{split}
\int\left(\frac{\varphi(\gamma_0)-\varphi(\gamma_t)}{\sfd(\gamma_0,\gamma_t)}\right)^2\,\d\tilde\ppi(\gamma)
&\leq\int\frac1{t^2}\left(\int_0^t\weakgrad\varphi(\gamma_s)\,\d s\right)^2\,\d\tilde\ppi(\gamma)\leq
\frac1t\iint_0^t\weakgrad\varphi^2(\gamma_s)\,\d s\,\d\tilde\ppi(\gamma)\\
&=\frac1t\iint_0^t\weakgrad\varphi^2\,\d s\,\d(\e_t)_\sharp\tilde\ppi
=\frac1t\iint_0^t\weakgrad\varphi^2\frho_s\,\d s\,\d\mm,
\end{split}
\end{equation}
where $\frho_s$ is the density of $(\e_s)_\sharp\tilde\ppi$. Since
$(\e_t)_\sharp\tilde\ppi$ weakly converges to $(\e_0)_\sharp\tilde\ppi$ as
$t\downarrow 0$ and $\entr{(\e_t)_\sharp\tilde\ppi}\mm$ is uniformly bounded
(by the $K$-geodesic convexity), we conclude that $\frho_t\to\frho_0$ weakly in $L^1(X,\mm)$ and
since $\weakgrad\varphi\in L^\infty(X,\mm)$ we have
\begin{equation}
\label{eq:finebrenier}
\lim_{t\downarrow 0}\frac1t\iint_0^t\weakgrad\varphi^2\frho_s\,\d s\,\d\mm=\intX\weakgrad\varphi^2\frho_0\,\d\mm.
\end{equation}
Equations \eqref{eq:prelimite}, \eqref{eq:intermediobrenier} and \eqref{eq:finebrenier} yield \eqref{eq:perbre}.

In order to prove \eqref{eq:perbre} in the general case of possibly
unbounded densities, let us 
fix a Kantorovich potential
$\varphi$,  $\ppi\in\gopt(\frho_0\mm,\frho_1\mm)$ and for $n\in\N$
define
$\ppi^n:=c_n\ppi\restr{\{\gamma:\frho_0(\gamma_0)+\frho_1(\gamma_1)\leq
n\}}$, $c_n\to 1$ being the normalization constant. Then
$\ppi^n\in\gopt(\frho^n_0\mm,\frho^n_1\mm)$, where
$\frho^n_i:=(\e_i)_\sharp\ppi^n$, $\varphi$ is a Kantorovich
potential for $(\frho^n_0\mm,\frho^n_1\mm)$ and $\frho^n_0,\frho^n_1\in
L^\infty(X,\mm)$. Thus from what we just proved we know that it
holds
\[
\sfd(\gamma_0,\gamma_1)=\weakgrad\varphi(\gamma_0)=|D^+\varphi|(\gamma_0),\qquad
\qquad\text{for $\ppi^n$-a.e.~$\gamma$.}
\]
Letting $n\to\infty$ we conclude.

Concerning \eqref{eq:6},  we can choose $\tilde\ppi=\ppi$ and
obtain by \eqref{eq:7} and \eqref{eq:8}
\begin{displaymath}
  \frac{\varphi(\gamma_0)-\varphi(\gamma_t)}{\sfd(\gamma_0,\gamma_t)}\ge0,\qquad
  \liminf_{t\downarrow0}\frac{\varphi(\gamma_0)-\varphi(\gamma_t)}{\sfd(\gamma_0,\gamma_t)}\ge
  |D^+\varphi|(\gamma_0)\quad
  \text{for $\ppi$-a.e.\ $\gamma$.}    
\end{displaymath}
On the other hand \eqref{eq:intermediobrenier} and \eqref{eq:finebrenier} yield
\begin{displaymath}
  \limsup_{t\downarrow0}\int
  \Big(\frac{\varphi(\gamma_0)-\varphi(\gamma_t)}{\sfd(\gamma_0,\gamma_t)}\Big)^2\,\d\ppi(\gamma)\le 
  \int |D^+\varphi|^2(\gamma_0)\,\d\ppi(\gamma),
\end{displaymath}
so that, by expanding the square and applying Fatou's Lemma, we obtain
\begin{align*}
  \limsup_{t\downarrow0}&\int
  \Big(\frac{\varphi(\gamma_0)-\varphi(\gamma_t)}{\sfd(\gamma_0,\gamma_t)}-|D^+\varphi|(\gamma_0)\Big)^2\,\d\ppi(\gamma)\le0.
\end{align*}
\end{proof}
\section{More on calculus on compact $CD(K,\infty)$ spaces}\label{se:calculus}

\subsection{On horizontal and vertical derivatives again}

Aim of this subsection is to prove another deep relation between
``horizontal'' and ``vertical'' derivation, which will allow to
compare the derivative of the squared Wasserstein distance along the
heat flow with the derivative of the relative entropy along a
geodesic (see the next subsection). This will be key in order to
understand the properties of spaces with \emph{Riemannian} Ricci
curvature bounded from below, illustrated in the last section.

In order to understand the geometric point, consider the following simple example.

\begin{example}{\rm
Let $\|\cdot\|$ be a smooth, strictly convex norm on $\R^d$ and let
$\|\cdot\|_*$ be the dual norm. Denoting by $\langle
\cdot,\cdot\rangle$ the canonical duality from $(\R^d)^*\times\R^d$
into $\R$, let $\mathcal L$ be the duality map from
$(\R^d,\|\cdot\|)$ to $((\R^d)^*,\|\cdot\|_*)$, characterized by
$$
\langle\mathcal L(u),u\rangle=\|\mathcal
L(u)\|_*\|u\|\quad\text{and}\quad \|\mathcal L(u)\|_*=\|u\|
\qquad\forall u\in \R^d,
$$
and let $\mathcal L^*$ be its inverse, equally characterized by
$$
\langle v,\mathcal L^*(v)\rangle=\|v\|_*\|\mathcal
L^*(v)\|\quad\text{and}\quad \|\mathcal
L^*(v)\|=\|v\|_*\qquad\forall v\in (\R^d)^*.
$$
Using the fact that $\epsilon\mapsto \|u\|\|u+\epsilon u'\|-\langle
\mathcal L u,u+\epsilon u'\rangle$ attains its minimum at
$\epsilon=0$ and the analogous relation for $\mathcal L^*$, one
obtains the useful relations
\begin{equation}\label{eq:normasquare}
\langle\mathcal L(u),u'\rangle=\frac 12\d_u\|\cdot\|^2(u'),\qquad
\langle v',\mathcal L^*(v)\rangle=\frac 12\d_v\|\cdot\|_*^2(v').
\end{equation}
For a smooth map $f:\R^d\to\R$ its differential $\d_xf$ at any point
$x$ is intrinsically defined as cotangent vector, namely as an
element of $(\R^d)^*$. To define the gradient $\nabla f(x)\in\R^d$
(which is a tangent vector), the norm comes into play via the
formula $\nabla f(x):=\mathcal L^*(\d_xf)$. Now, given two smooth
functions $f,\,g$, the real number $\d_xf(\nabla g(x))$ is well
defined as the application of the cotangent vector $\d_xf$ to the
tangent vector $\nabla g(x)$. 

What we want to point out, is that
there are two very different ways of obtaining $\d_xf(\nabla g(x))$
from a derivation. The first one,
which is usually taken as the definition of
$\d_xf(\nabla g(x))$, is the ``horizontal derivative'':
\begin{equation}
\label{eq:hor}
\langle \d_x f,\nabla g\rangle =\d_xf(\nabla g(x))=\lim_{t\to
0}\frac{f(x+t\nabla g(x))-f(x)}t.
\end{equation} 
The second one is the ``vertical derivative'':
\begin{equation}
\label{eq:ver}
Df(\nabla g)(x)=\lim_{\eps\to0}\frac{\frac12\|\d_x(g+\eps f)\|^2_*-\frac12\|\d_x g\|_*^2(x)}{\eps}.
\end{equation}
It is not difficult to check that \eqref{eq:ver} is consistent with
\eqref{eq:hor}: indeed (omitting the $x$ dependence), recalling the
second identity of \eqref{eq:normasquare}, we have
\begin{displaymath}
 \|\d g+\eps \d f\|^2_*=\|\d g\|_*^2+
2\eps \langle \mathcal L^*(\d g),\d f\rangle+o(\eps)=
\|\nabla g\|^2+2\eps \langle \nabla g,\d f\rangle+o(\eps).
\end{displaymath}
}
\fr\end{example}
The point is that the equality between the right hand sides of
formulas \eqref{eq:ver} and \eqref{eq:hor} extends to a genuine
metric setting. In the following lemma (where the plan $\ppi$ plays
the role of $-\nabla g$) we prove one inequality, but we remark that
``playing with signs'' it is possible to obtain an analogous
inequality with $\leq$ in place of $\geq$.

\begin{lemma}[Horizontal and vertical derivatives]\label{le:horver}
Let $f$ be a Sobolev function along a.e.\ curve with
$\weakgrad f\in
L^2(X,\mm)$, let $g:X\to\R$ be Lipschitz and let $\ppi$ be a test plan
concentrated on $\geo(X)$. Assume that 
\begin{equation}
\label{eq:asshorvert} 
\lim_{t\downarrow0}\frac{g(\gamma_0)-g(\gamma_t)}
{\sfd(\gamma_0,\gamma_t)}=\weakgrad
g(\gamma_0)
\qquad\text{in }L^2(\mathrm{Geo}(X),\ppi).
\end{equation}
Then
\begin{equation}\label{eq:derhorvert}
\limi_{t\downarrow0}\int\frac{f(\gamma_{t})-f(\gamma_0)}{t}\,\d\ppi(\gamma)
\geq\frac12\int\frac{\weakgrad g^2(\gamma_0)-\weakgrad{(g+\eps
f)}^2(\gamma_0)}\eps\, \d\ppi(\gamma)\qquad\forall\eps>0.
\end{equation}
\end{lemma}
\begin{proof}
Define the functions $F_t,\,G_t:\geo(X)\to\R\cup\{\pm\infty\}$ by
\[
\begin{split}
F_t(\gamma)&:=\frac{f(\gamma_0)-f(\gamma_t)}{\sfd(\gamma_0,\gamma_t)},\\
G_t(\gamma)&:=\frac{g(\gamma_0)-g(\gamma_t)}{\sfd(\gamma_0,\gamma_t)}.
\end{split}
\]
By \eqref{eq:asshorvert} 
it holds
\begin{equation}
\label{eq:uguale}
\int \weakgrad g^2\circ\e_0\,\d\ppi(\gamma)=\lim_{t\downarrow 0}\int  G_t^2\,\d\ppi.
\end{equation}
Since the measures $(\e_t)_\sharp\ppi\to(\e_0)_\sharp\ppi$ weakly in
duality with  $C(X)$ as
$t\downarrow 0$ and their densities with respect to $\mm$ are uniformly bounded, we obtain
that the densities are weakly$^*$ convergent in $L^\infty(X,\mm)$.
Therefore, using the fact that $\weakgrad{(g+\eps f)}^2\in
L^1(X,\mm)$ and taking into account Remark~\ref{re:restr} we obtain
\[
\begin{split}
\int&\weakgrad{(g+\eps f)}^2\circ\e_0\,\d\ppi(\gamma)=\int\weakgrad{(g+\eps f)}^2\,\d(\e_0)_\sharp\ppi
=\lim_{t\downarrow 0}\frac1t\int_0^t\int_X\weakgrad{(g+\eps f)}^2\,\d(\e_s)_\sharp\ppi\,\d s\\
&=\lim_{t\downarrow 0}\frac1t\iint_0^t\weakgrad{(g+\eps f)}^2(\gamma_s)\,\d s\,\d\ppi(\gamma)\geq
\lims_{t\downarrow 0}\int\left|\frac{(g+\eps f)(\gamma_0)-(g+\eps f)(\gamma_{t})}
{t\sfd(\gamma_0,\gamma_1)}\right|^2\,\d\ppi(\gamma)\\
&\geq \lims_{t\downarrow 0}\int G_t^2+2\eps G_tF_t \,\d\ppi.
\end{split}
\]
Subtracting this inequality from \eqref{eq:uguale} and dividing by
$2\eps$ we get
\[
\frac12\int \frac{\weakgrad g^2(\gamma_0)-\weakgrad{(g+\eps
f)}^2(\gamma_0)}\eps \,\d\ppi(\gamma)\leq\limi_{t\downarrow 0}-\int
G_t(\gamma)F_t(\gamma)\,\d\ppi(\gamma).
\]
We know that $G_t\to \weakgrad g\circ\e_0$ in $L^2(\geo(X),\ppi)$
and that $\weakgrad g(\gamma_0)=\sfd(\gamma_0,\gamma_1)$ for
$\ppi$-a.e.~$\gamma$. Also, by Remark~\ref{re:restr} and the fact
that $\ppi$ is a test plan we easily get
$\sup_{t\in[0,1]}\|F_t\|_{L^2(\sppi)}<\infty$. Thus it holds
\[
\begin{split}
\limi_{t\downarrow 0}-\int G_t(\gamma)F_t(\gamma)\,\d\ppi(\gamma)
&=\limi_{t\downarrow 0}-\int\sfd(\gamma_0,\gamma_1)F_t(\gamma)\,\d\ppi(\gamma)
=\limi_{t\downarrow 0}\int\frac{f(\gamma_t)-f(\gamma_0)}{t}\,\d\ppi(\gamma),
\end{split}
\]
which is the thesis.
\end{proof}

\subsection{Two important formulas}

\renewcommand{\ssigma}{\sigma}
\begin{proposition}[Derivative of $\frac12W_2^2$ along the heat flow]\label{prop:derw2}
Let $(\frho_t)\subset L^2(X,\mm)$ be a heat flow made of probability
densities. Then for every $\ssigma\in\prob X$, for a.e.~$t\in
(0,\infty)$ it holds:
\begin{equation}
\label{eq:derivataw2} \frac{\d}{\d
t}\frac12W_2^2(\frho_t\mm,\ssigma)=\intX\varphi_t\Delta\frho_t\,\d\mm,\qquad
\text{for any Kantorovich potential $\varphi$ from $\frho_t$ to
$\ssigma$.}
\end{equation}
\end{proposition}
\begin{proof}
Since $t\mapsto \frho_t\mm$ is an absolutely continuous curve w.r.t.
$W_2$ (recall Theorem~\ref{thm:main}), the derivative at the left
hand side of \eqref{eq:derivataw2} exists for a.e.~$t\in
(0,\infty)$. Also, for a.e.~$t\in (0,\infty)$ it holds $\lim_{h\to
0}\frac 1h(\frho_{t+h}-\frho_t)=\Delta\frho_t$, the limit being
understood in $L^2(X,\mm)$.

Fix $t_0$ such that the derivative of the Wasserstein distance
exists and the above limit holds and choose any Kantorovich
potential $\varphi_{t_0}$ for $(\frho_{t_0}\mm,\ssigma)$. We have
\[
\begin{split}
\frac{W_2^2(\frho_{t_0}\mm,\ssigma)}2&=\int_X\varphi_{t_0}\frho_{t_0}\,\d\mm+\int\varphi_{t_0}^c\,\d\ssigma\\
\frac{W_2^2(\frho_{t_0+h}\mm,\ssigma)}2&\geq\int_X\varphi_{t_0}\frho_{t_0+h}\,\d\mm+\int\varphi_{t_0}^c\,\d\ssigma.
\end{split}
\]
Therefore, since $\varphi_{t_0}\in L^\infty(X,\mm)$ we get
\[
\frac{W_2^2(\frho_{t_0+h}\mm,\ssigma)}2-\frac{W_2^2(\frho_{t_0}\mm,\ssigma)}2
\geq\intX\varphi_{t_0}(\frho_{t_0+h}-\frho_{t_0})\,\d\mm=
h\int_X\varphi_{t_0}\Delta\frho_{t_0}+o(h).
\]
Dividing by $h<0$ and $h>0$ and letting $h\to 0$ we get the thesis.
\end{proof}

\begin{proposition}[Derivative of the Entropy along a geodesic]\label{prop:derentr}
Let $(X,\sfd,\mm)$ be a strong $CD(K,\infty)$ space. Let
$\mu_0,\,\mu_1\in\prob X$, $\ppi\in\gopt(\mu_0,\mu_1)$ and $\varphi$
a Kantorovich potential for $(\mu_0,\mu_1)$. Assume that $\ppi$ is a
test plan and that $\mu_0\geq c\mm$ from some $c>0$ and denote by
$\hsigma_t$ the density of $\mu_t:=(\e_t)_\sharp\ppi$. Then
\begin{equation}
\label{eq:derentropia}
\limi_{t\downarrow 0}\frac{\entr{\mu_t}\mm-\entr{\mu_0}\mm}{t}\geq
\lim_{\eps\downarrow 0}\frac{\C(\varphi)-\C(\varphi+\eps\hsigma_0)}{\eps}
\end{equation}
\end{proposition}
\begin{proof}
The convexity of $\C$ ensures that the limit at the right hand side
exists. From the fact that $\varphi$ is Lipschitz, it is not hard to
see that $\hsigma_0\notin D(\C)$ implies
$\C(\varphi+\eps\hsigma_0)=+\infty$ for any $\eps>0$ and in this case
there is nothing to prove. Thus, we assume that $\hsigma_0\in D(\C)$.

The convexity of $z\mapsto z\log z$ gives
\begin{equation}
\label{eq:fine1}
\begin{split}
\frac{\entr{\mu_t}\mm-\entr{\mu_0}\mm}t&\geq\int_X\log\hsigma_0\frac{\hsigma_t-\hsigma_0}t\,\d\mm=\int
\frac{\log(\hsigma_0\circ e_t)-\log(\hsigma_0\circ e_0)}t \,\d\ppi.
\end{split}
\end{equation}
Using the trivial inequality given by Taylor's formula
\[
\log b-\log a\geq\frac{b-a}a-\frac{|b-a|^2}{2c^2},
\]
valid for any $a,\,b\in [c,\infty)$, we obtain
\begin{equation}
\label{eq:fine2}
\begin{split}
\int&\frac{\log(\hsigma_0\circ \e_t)-\log(\hsigma_0\circ \e_0)}t
\,\d\ppi\geq\int\frac{\hsigma_0\circ
e_t-\hsigma_0\circ\e_0}{t\hsigma_0\circ
\e_0}\,\d\ppi-\frac1{2tc^2}\int|\hsigma_0\circ\e_t-\hsigma_0\circ\e_0|^2\,\d\ppi.
\end{split}
\end{equation}
Taking into account Remark~\ref{re:restr} and the fact that 
$|\dot \gamma_t|=\sfd(\gamma_0,\gamma_1)\le \mathrm{diam}(X)$ 
for a.e.\ $t\in (0,1)$ and $\ppi$-a.e.\ $\gamma$, the last term in this
expression can be bounded from above by
\begin{equation}
\label{eq:fine3}
\begin{split}
\frac1{2tc^2}\int\biggl(\int_0^t  \mathrm{diam}(X)\weakgrad{\hsigma_0}\circ\e_s\biggr)^2
\,\d
s\,\d\ppi\leq\frac{\mathrm{diam}(X)^2}{2c^2}\int\int_0^t\weakgrad{\hsigma_0}^2\circ\e_s\,\d
s\,\d\ppi,
\end{split}
\end{equation}
which goes to 0 as $t\to 0$.

Now let $S:\geo(X)\to\R$ be the Borel function defined by
$S(\gamma):=\hsigma_0\circ\gamma_0$ and define $
\tilde\ppi:=\frac1S\ppi. $ It is easy to check that
$(\e_0)_\sharp \tilde\ppi=\mm$, so that in particular $\tilde\ppi$ is a
probability measure. Also, the bound $\hsigma_0\geq c>0$ ensures that
$\tilde\ppi$ is a test plan. By definition we have
\[
\int\frac{\hsigma_0\circ e_t-\hsigma_0\circ\e_0}{t\hsigma_0\circ \e_0}\,\d\ppi
=\int\frac{\hsigma_0\circ e_t-\hsigma_0\circ\e_0}{t}\,\d\tilde \ppi.
\]
The latter equality and inequalities \eqref{eq:fine1},
\eqref{eq:fine2} and \eqref{eq:fine3} ensure that to conclude it is
sufficient to show that
\begin{equation}
\label{eq:perfinire}
\limi_{t\downarrow0}\int \frac{\hsigma_0\circ e_t-\hsigma_0\circ\e_0}{t}\,\d\tilde \ppi
\geq \lim_{\eps\downarrow 0}\frac{\C(\varphi)-\C(\varphi+\eps\hsigma_0)}{\eps}.
\end{equation}
Here we apply the key Lemma~\ref{le:horver}. Observe that
Theorem~\ref{prop:potweak} ensures that
\[
\weakgrad\varphi(\gamma_0)=\lim_{t\downarrow
0}\frac{\varphi(\gamma_0)-\varphi(\gamma_t)}{t}=\sfd(\gamma_0,\gamma_1)
\]
where the convergence is understood in $L^2(\ppi)$. Thus the same
holds for $L^2(\tilde\ppi)$ and the hypotheses of
Lemma~\ref{le:horver} are satisfied with $\tilde\ppi$ as test plan,
$g:=\varphi$ and $f:=\hsigma_0$. Equation \eqref{eq:derhorvert} then
gives
\[
\begin{split}
\limi_{t\downarrow0}\int \frac{\hsigma_0\circ e_t-\hsigma_0\circ\e_0}{t}\,\d\tilde \ppi&\geq
\lims_{\eps\downarrow 0}\frac12\int \frac{\weakgrad {
    \varphi}^2(\gamma_0)-\weakgrad{(\varphi+\eps \hsigma_0)}^2(\gamma_0)}
\eps\,\d\tilde\ppi(\gamma)\\
&=\lims_{\eps\downarrow 0}\frac12\int_X\frac{\weakgrad {\varphi}^2(x)-\weakgrad{(\varphi+\eps\hsigma_0)}^2(x)}\eps \,\d\mm(x),
\end{split}
\]
which concludes the proof.
\end{proof}

\section{Riemannian Ricci bounds}\label{se:riemricc}

We say that $(X,\sfd,\mm)$ has \emph{Riemannian Ricci curvature}
bounded below by $K\in\R$ (in short, it is a $RCD(K,\infty)$ space)
if any of the 3 equivalent conditions stated in the following
theorem is true.

\begin{theorem}\label{thm:riemannian}
Let $(X,\sfd,\mm)$ be a compact and normalized metric measure space
and $K\in\R$. The following three properties are equivalent.
\begin{itemize}
\item[(i)] $(X,\sfd,\mm)$ is a strong $CD(K,\infty)$ space (Definition~\ref{def:strongcd})
and the $L^2$-gradient flow of $\C$ is linear. 
\item[(ii)] $(X,\sfd,\mm)$ is a strong $CD(K,\infty)$ space (Definition~\ref{def:strongcd}) and Cheeger's
energy is quadratic, i.e.
\begin{equation}
\label{eq:cquadr}
2\big(\C(f)+\C(g)\big)=\C(f+g)+\C(f-g),\qquad\forall f,\,g\in
L^2(X,\mm).
\end{equation}
\item[(iii)]  $\supp(\mm)$ 
is geodesic and for any $\mu\in D(\entv)\subset
\prob X$ 
 there exists an ${\rm EVI}_K$-gradient flow for $\entv$ starting from $\mu$.
\end{itemize}
\end{theorem}
\begin{proof}\\*
${\mathbf{ (i)\Rightarrow (ii)}}$. Since the heat semigroup $P_t$ in
$L^2(X,\mm)$ is linear we obtain that $\Delta$ is a linear operator
(i.e. its domain $D(\Delta)$ is a subspace of $L^2(X,\mm)$ and
$\Delta:D(\Delta)\to L^2(X,\mm)$ is linear). Since $t\mapsto
\C(P_t(f))$ is locally Lipschitz, tends to 0 as $t\to\infty$ and
$\partial_t\C(P_t(f))=-\|\Delta P_t(f)\|^2_{L^2}$ for a.e.~$t>0$
(see \eqref{eq:perdopo}), we have
\[
\C(f)=\int_0^\infty\|\Delta P_t(f)\|^2_{L^2(X,\smm)}\,\d t.
\]
Therefore $\C$, being an integral of quadratic forms, is a quadratic
form. Specifically, for any $f,\,g\in L^2(X,\mm)$ it holds
\[
\begin{split}
\C(f+g)+\C(f-g)&=\int_0^\infty\|\Delta P_t(f+g)\|^2_{L^2(X,\smm)}+\|\Delta P_t(f-g)\|^2_{L^2(X,\smm)}\,\d t\\
&=\int_0^\infty\|\Delta P_t(f)+\Delta P_t(g)\|^2_{L^2(X,\smm)}+\|\Delta P_t(f)-\Delta P_t(g)\|^2_{L^2(X,\smm)}\,\d t\\
&=\int_0^\infty2\|\Delta P_t(f)\|^2_{L^2(X,\smm)}+2\|\Delta P_t(g)\|^2_{L^2(X,\smm)}\,\d t\\
&=2\C(f)+2\C(g).
\end{split}
\]
${\mathbf{(ii)\Rightarrow (iii)}}$. 
By \cite[Remark 4.6(iii)]{Sturm06I}
$(\supp(\mm),\sfd)$ is a length space 
and therefore it is also
geodesic, since $X$ is compact.

Thanks to Remark~\ref{re:contr}
it is sufficient to prove that a gradient flow in the ${\rm EVI}_K$
sense exists for an initial datum $\mu_0\ll\mm$ with density bounded
away from 0 and infinity. Let $\frho_0$ be this density, $(\frho_t)$
the heat flow starting from it and recall that from the maximum
principle \ref{prop:basecal} we know that the $\frho_t$'s are far
from 0 and infinity as well for any $t>0$. Fix a reference
probability measure $\ssigma$ with density bounded away from 0 and
infinity as well. For any $t\geq 0$ pick a test plan $\ppi_t$
optimal for $(\frho_t\mm,\ssigma)$. Define
$\ssigma_t^s:=(\rme_s)_\sharp \pi_t$.

We claim that for a.e.~$t\in (0,\infty)$ it holds
\begin{equation}
\label{eq:evilocal} \frac{\d}{\d
t}\frac12W_2^2(\frho_t\mm,\sigma\mm)\leq \limi_{s\downarrow
0}\frac{\ent{\sigma_t^s}-\ent{\sigma_t^0}}s.
\end{equation}
Let $\varphi_t$ be a Kantorovich potential for
$\frho_t\mm,\sigma\mm$. By Proposition~\ref{prop:derw2} we know that
for a.e.~$t\in (0,\infty)$ it holds
\[
\frac{\d}{\d t}\frac12W_2^2(\frho_t\mm,\sigma\mm)
=\intX\varphi\Delta\frho_t\,\d\mm\leq \lim_{\eps\downarrow 0}\frac{\C(\frho_t-\eps\varphi_t)-\C(\frho_t)}{\eps},
\]
while from Proposition~\ref{prop:derentr} we have that for any $t>0$
it holds
\[
\limi_{s\downarrow 0}\frac{\ent{\sigma_t^s}-\ent{\sigma_t^0}}s
\geq \lim_{\eps\downarrow 0}\frac{\C(\varphi_t)-\C(\varphi_t+\eps\frho_t)}{\eps}.
\]
Here we use the fact that $\C$ is quadratic. Indeed in this case simple algebraic manipulations show that
\[
\frac{\C(\frho_t-\eps\varphi_t)-\C(\frho_t)}{\eps}=\frac{\C(\varphi_t)-\C(\varphi_t+\eps\frho_t)}{\eps}
+O(\epsilon),\qquad\forall t>0,
\]
and therefore \eqref{eq:evilocal} is proved.

Now notice that the $K$-convexity of the entropy yields
\[
\limi_{s\downarrow 0}\frac{\ent{\sigma_t^s}-\ent{\sigma_t^0}}s\leq
\ent\ssigma-\ent{\frho_t\mm}-\frac K2W_2^2(\frho_t\mm,\ssigma),
\]
and therefore we have
\[
\frac{\d}{\d
t}\frac12W_2^2(\frho_t\mm,\sigma\mm)+\ent{\frho_t\mm}+\frac
K2W_2^2(\frho_t\mm,\ssigma)\leq\ent\ssigma,\qquad \text{for a.e.
$t\in (0,\infty)$.}
\]
By Proposition~\ref{prop:equivevi} we conclude.\\*

${\mathbf{ (iii)\Rightarrow (i)}}$. Since $(\supp(\mm),\sfd)$ is geodesic, so
is $(\overline{D(\entv)},W_2)$, 
which together with existence of ${\rm
EVI}_K$-gradient flows for $\entv$ yields, via
Proposition~\ref{prop:dansav}, $K$-geodesic convexity of $\entv$
along all geodesics in $\overline{D(\entv)}$. 
In particular, $(X,\sfd,\mm)$ is a
strong $CD(K,\infty)$ space.

We turn to the linearity. Let $(\mu_t^0)$, $(\mu^1_t)$ be two ${\rm
EVI}_K$-gradient flows of the relative entropy and, for $\lambda\in
(0,1)$ fixed, define
$\mu^\lambda_t:=(1-\lambda)\mu^0_t+\lambda\mu^1_t$.

We claim that $(\mu_t)$ is an ${\rm EVI}_K$-gradient flow of
$\entv$. To prove this, fix $\nu\in\prob X$, $t>0$ and an optimal
plan $\ggamma\in\opt(\mu_t^\lambda,\nu)$. Since
$\mu^i_t\ll\mu^\lambda_t=\pi^1_\sharp\ggamma$ for $i=0,1$ we can
define, as in Definition~\ref{def:pushplan}, the plans
$\ggamma_{\mu^i_t}\in\prob{X^2}$ and the measures
$\nu^i:=\ggamma_\sharp\mu^i_t$, $i=0,1$. Since
$\supp(\gamma_{\mu^i_t})\subset\supp(\ggamma)$, we have that
$\gamma_{\mu^i_t}\in\opt(\mu^i_t,\nu^i)$, therefore from
$\ggamma=(1-\lambda)\ggamma_{{\mu^0_t}}+\lambda\ggamma_{\mu^1_t}$ we
deduce
\begin{equation}
\label{eq:distugualeevi}
W_2^2(\mu_t^\lambda,\nu)=(1-\lambda)W_2^2(\mu^0_t,\nu^0)+\lambda
W_2^2(\mu^1_t,\nu^1).
\end{equation}
On the other hand, from the convexity of the squared Wasserstein
distance we immediately get that
\begin{equation}
\label{eq:distdivevi}
W_2^2(\mu^\lambda_{t+h},\nu)\leq(1-\lambda) W_2^2(\mu^0_{t+h},\nu^0)
+\lambda W_2^2(\mu^1_{t+h},\nu^1),\qquad\forall h>0.
\end{equation}
Furthermore, recalling $(iii)$ of Proposition~\ref{prop:pushgamma},
we get
\begin{equation}
\label{eq:convent}
\entr{\mu^\lambda_t}\mm-\entr\nu\mm\leq(1-\lambda)\big(\entr{\mu^0_t}\mm-\entr{\nu^0}\mm\big)
+\lambda\big(\entr{\mu^1_t}\mm-\entr{\nu^1}\mm\big).
\end{equation}
The fact that $(\mu^0_t)$ and $(\mu^1_t)$ are ${\rm EVI}_K$-gradient
flows for $\entv$ (see in particular the characterization $(iii)$
given in Proposition~\ref{prop:equivevi}) in conjunction with
\eqref{eq:distugualeevi}, \eqref{eq:distdivevi} and
\eqref{eq:convent} yield
\begin{equation}
\label{eq:evilin}
\lims_{h\downarrow 0}\frac{W_2^2(\mu^\lambda_{t+h},\nu)-W_2^2(\mu^\lambda_t,\nu)}2
+\frac K2W_2^2(\mu^\lambda_t,\nu)+\entr{\mu^\lambda_t}\mm\leq\entr\nu\mm.
\end{equation}
Since $t>0$ and $\nu\in\prob X$ were arbitrary, we proved that
$(\mu^\lambda_t)$ is a ${\rm EVI}_K$-gradient flow of $\entv$ (see
again $(iii)$ of Proposition~\ref{prop:equivevi}).

Thus, recalling the identification of gradient flows, we proved that
the $L^2$-heat flow is additive in $D(\entv)$. Since the heat flow
in $L^2(X,\mm)$ commutes with additive and multiplicative constants,
it is easy to get from this linearity in the class of bounded
functions. By $L^2$ contractivity, linearity extends to the whole of
$L^2(X,\mm)$.
\end{proof}

We conclude by discussing some basic properties of the spaces with
Riemannian Ricci curvature bounded from below.

We start observing that Riemannian manifolds with Ricci curvature
bounded below by $K$ are $RCD(K,\infty)$ spaces, as they are non
branching $CD(K,\infty)$ spaces and the heat flow is linear on them.
Also, from the studies made in \cite{Petrunin10}, \cite{ZZ1},
\cite{Ohta09} and \cite{Gigli-Ohta10} we also know that finite
dimensional Alexandrov spaces with curvature bounded from below are
$RCD(K,\infty)$ spaces as well. On the other side, Finsler manifolds
are ruled out, as it is known (see for instance
\cite{Ohta-Sturm09}) that the heat flow is linear on a Finsler
manifold if and only if the manifold is Riemannian.

The stability of the $RCD(K,\infty)$ notion can be deduced by the
stability of ${\rm EVI}_K$-gradient flows w.r.t.
$\Gamma$-convergence of functionals, which is an easy consequence of
the integral formulation in $(ii)$ of
Proposition~\ref{prop:equivevi}.

Hence $RCD(K,\infty)$ spaces have the same basic properties of $CD(K,\infty)$ spaces,
which gives to this notion the right of being called a synthetic (or weak) notion of Ricci curvature bound.

The point is then to understand the additional analytic/geometric
properties of these spaces, which come mainly by the addition of
linearity condition. A first consequence is that the heat flow
contracts, up to an exponential factor, the distance $W_2$, i.e.
\[
W_2(\mu_t,\nu_t)\leq e^{-Kt} W_2(\mu_0,\nu_0),\qquad\forall t\geq 0,
\]
whenever $(\mu_t),\,(\nu_t)\subset\probt X$ are gradient flows of
the entropy.

By a duality argument (see \cite{Kuwada10}, \cite{GigliKuwadaOhta10},
\cite{Ambrosio-Gigli-Savare11bis}), this property implies the
Bakry-Emery gradient estimate
\[
\weakgrad{{\sf h}_t(f)}^2(x)\leq e^{-2Kt}{\sf
h}_t(\weakgrad{f}^2)(x),\qquad\text{for $\mm$-a.e.~$x\in X$,}
\]
for all $t>0$, where ${\sf h}_t:L^2(X,\mm)\to L^2(X,\mm)$ is the
heat flow seen as gradient flow of $\C$. If $(X,\sfd,\mm)$ is
doubling and supports a local Poincar\'e inequality, then also the
Lipschitz regularity of the heat kernel is deduced (following an
argument described in \cite{GigliKuwadaOhta10}).

Also, since in $RCD(K,\infty)$ spaces $\C$ is a quadratic form, if
we define
\[
\mathcal E(f,g):=\C(f+g)-\C(f)-\C(g),\qquad\forall f,g\in W^{1,2}(X,\sfd,\mm),
\]
we get a closed Dirichlet form on $L^2(X,\mm)$ (closure follows from
the $L^2$-lower semicontinuity of $\C$). Hence it is natural to
compare the calculus on $RCD(K,\infty)$ spaces with the abstract one
available for Dirichlet forms (see \cite{Fukushima80}). The picture
here is pretty clear and consistent.  Recall that to any $f\in
D(\mathcal E)$ one can associate the energy measure $[f]$ defined by
\[
[f](\varphi):=-\mathcal E(f,f\varphi)+\mathcal E(f^2/2,\varphi).
\]
Then it is possible to show that the energy measure coincides with
$\relgrad f^2\mm$. Also, the distance $\sfd$ coincides with the
intrinsic distance $\sfd_{\mathcal E}$ induced by the form, defined
by
\[
\sfd_{\mathcal E}(x,y):=\sup\Big\{| g(x)-g(y)|\ :\ g\in D(\mathcal E)\cap C(X),\ [g]\leq\mm\Big\}.
\]
Taking advantage of these identification and of the locality of
$\mathcal E$ (which is a consequence of the locality of the notion
$\relgrad f$), one can also see that on $RCD( K,\infty)$ spaces a
continuous Brownian motion with continuous sample paths associated
to ${\sf h}_t$ exists and is unique.

Finally, for $RCD(K,\infty)$ spaces it is possible to prove tensorization
and globalization properties which are in line with those available for $CD(K,\infty)$ spaces.

\def\cprime{$'$}

\end{document}